\documentclass[a4paper,english,12pt]{article}

\usepackage{geometry}
\usepackage{pdfpages}
\usepackage[T1]{fontenc}
\usepackage[utf8]{inputenc}
\usepackage{babel}
\usepackage{amsmath}
\usepackage{amssymb}
\usepackage{amsthm}
\usepackage{textcomp}
\usepackage{enumitem}
\usepackage{dsfont}
\usepackage{mathrsfs}
\usepackage{cite}
\usepackage{color}
\usepackage[colorlinks=true,urlcolor=blue,linkcolor=blue]{hyperref}

\newcommand{\R}{\mathbb{R}}

\newcommand{\Z}{\mathbb{Z}}
\newcommand{\N}{\mathbb{N}}
\newcommand{\C}{\mathbb{C}}
\newcommand{\D}{\mathbb{D}}

\newcommand{\T}{\mathbb{T}}

\newcommand{\classeC}{\mathcal{C}}

\renewcommand{\Re}{\operatorname{Re}}
\newcommand{\un}{\mathds{1}}

\newcommand{\Res}{\operatorname{Res}}

\newcommand{\Tr}{\operatorname{Tr}}
\newcommand{\sinc}{\operatorname{sinc}}
\newcommand{\Gr}{\operatorname{Gr}}
\renewcommand{\mod}{\operatorname{mod}}

\newcommand{\hamilton}{\mathcal{H}}

\renewcommand\d{\,{\mathrm d}}

\newcommand{\gdO}{{\mathcal O}}
\newcommand{\longrightarroww}[2] {\mathop{\longrightarrow}\limits_{#1}^{#2}}
\newcommand{\weaklim}[2] {\mathop{\rightharpoonup}\limits_{#1}^{#2}}

\newtheorem{mydef}{Definition}[section]
\newtheorem{thm}[mydef]{Theorem}
\newtheorem{lem}[mydef]{Lemma}
\newtheorem{prop}[mydef]{Proposition}
\newtheorem{cor}[mydef]{Corollary}
\newtheorem{rk}[mydef]{Remark}

\title{Zero-dispersion limit for the Benjamin-Ono equation on the torus with single well initial data.}
\author{Louise Gassot}

\date{}
\newcommand{\Addresses}{{
  \bigskip
  \footnotesize
\noindent
  \textsc{ICERM, Brown University, 121 South Main Street, Providence, RI 02903, USA.}\par\nopagebreak
  \noindent
  \textit{E-mail address:} \texttt{louise\_gassot@brown.edu}
}}
\begin{document}
\maketitle
\abstract{
We consider the zero-dispersion limit for the Benjamin-Ono equation on the torus. We prove that when the initial data is a single well, the zero-dispersion limit exists in the weak sense and is uniform on every compact time interval. Moreover, the limit is equal to the signed sum of branches for the multivalued solution of the inviscid Burgers equation obtained by the method of characteristics. This result is similar to the one obtained  by Miller and Xu for the Benjamin-Ono equation on the real line for decaying and positive initial data. We also establish some precise asymptotics of the spectral data with initial data $u_0(x)=-\beta\cos(x)$, $\beta>0$,  justifying our approximation method, which is analogous to the work of Miller and Wetzel concerning a family of rational potentials for the Benjamin-Ono equation on the real line.}

\tableofcontents


\section{Introduction}


We consider the zero-dispersion limit $\varepsilon\to 0$ for the Benjamin-Ono equation
\begin{equation}\tag{BO-$\varepsilon$}\label{eq:bo}
\partial_t u
	=\partial_x(\varepsilon |\partial_{x}|u-u^2).
\end{equation}

The Benjamin-Ono equation~\cite{Benjamin1967,Ono1977} describes a certain regime of long internal waves in a two-layer fluid of great depth. The parameter $\varepsilon$ describes the balance between the dispersive term and the nonlinear term. For $\varepsilon=0$, this equation becomes the inviscid Burgers equation, which causes the formation of shocks in finite time. When $\varepsilon>0$, the dispersive term prevents the formation of a dispersive shock, replacing it by a train of waves (see for instance~\cite{MillerXu2011} for numerical simulations on the real line). It is expected that  right after the breaking time, the amplitude of shock waves does not decrease with $\varepsilon$, while the wavelength is proportional to $\varepsilon$. As a consequence, the limit can exist in the weak sense in the shock region.

\subsection{Zero-dispersion limit}
In this paper, we describe the weak zero-dispersion limit for the Benjamin-Ono equation at all times given a single well initial data $u_0\in\classeC^3(\T)$.

\begin{mydef}[Single well potential]\label{def:single well}
We say that $u_0\in\classeC^3(\T)$ is a single well potential if the following holds:
\begin{enumerate}
\item $u_0$ is real valued with zero mean;
\item $u_0(0)=\min_{x\in\T}u(x)$;
\item there exists $x_{\max}\in (0,2\pi)$ such that $u_0'>0$ on $(0,x_{\max})$ and $u_0'<0$ on $(x_{\max},2\pi)$;
\item there are exactly two inflection points $\xi_-\in(0,x_{\max})$ and $\xi_+\in(x_{\max},2\pi)$ such that $u_0''(\xi_{\pm})=0$, and the inflection points are simple $u_0'''(\xi_{\pm})\neq 0$.
\end{enumerate}
\end{mydef}
The latter condition only aims at simplifying the study of the Burgers equation with initial data $u_0$ in the proof of Theorem~\ref{thm:link_burgers}, and could be removed.

In this case, for every $\eta\in(\min u_0,\max u_0)$, $\eta$ has exactly two antecedents by $u_0$. We denote
\[
x_-(\eta)=\inf \{x\in[0,2\pi] \mid u_0(x)=\eta\},
\]
\[
x_+(\eta)=\sup \{x\in[0,2\pi] \mid u_0^(x)=\eta\}.
\]

 According to the work of Miller and Xu on the real line~\cite{MillerXu2011}, the relevant zero-dispersion limit for the Benjamin-Ono equation is the {\it multivalued} solution to the Burgers equation obtained by the method of characteristics. More precisely, every point $u^B$ is an image of this multivalued solution at time $t$ with abscissa $x$ as soon as it solves the implicit equation
\[
u^B=u_0(x-2u^Bt).
\]
Given $t$ and $x$, there may be several solutions $u^B$ that are denoted $u_0^B(t,x)<\dots<u_{2P(t,x)}^B(t,x)$, see Figure~\ref{fig:multivalued}. We define the signed sum of branches as
\begin{equation}\label{eq:u_alt}
u^B_{alt}(t,x):=\sum_{n=0}^{2P(t,x)}(-1)^nu_n^B(t,x).
\end{equation}
\begin{figure}
\begin{center}
\includegraphics[scale=0.3]{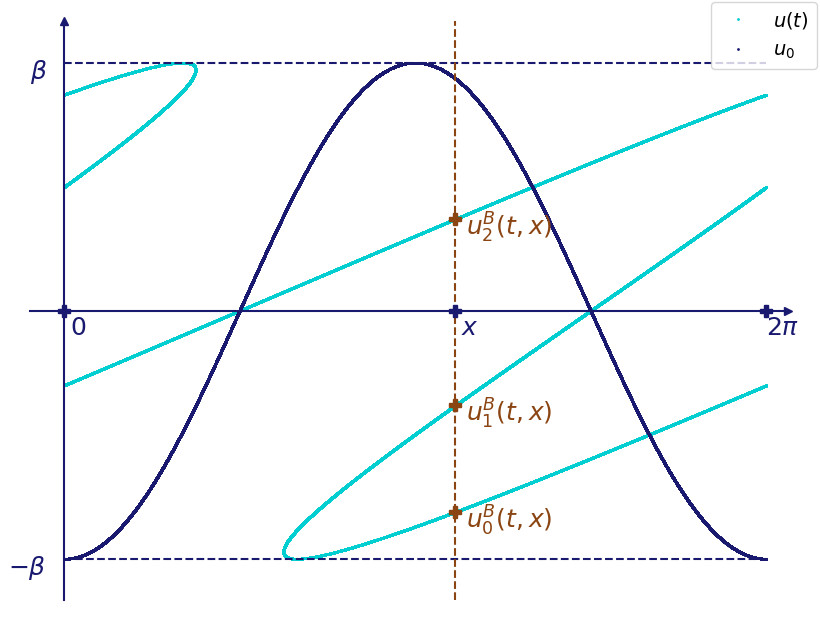}
\end{center}
\caption{Multivalued solution of the Burgers equation obtained by the method of characteristics, with initial data $u_0(x)=-\beta\cos(x)$}\label{fig:multivalued}
\end{figure}

Our main result is as follows.
\begin{thm}[Zero-dispersion limit for the Benjamin-Ono equation]\label{thm:main}
Let $u_0\in\classeC^3(\T)$ be a single well initial data. Let $u^B_{alt}$ be the signed sum of branches for the multivalued solution to the inviscid Burgers equation with initial data $u_0$. Then there exists a family $(u_0^{\varepsilon})_{\varepsilon>0}\subset L^2(\T)$ of initial data such that $u_0^{\varepsilon}\to u_0$ in $L^2(\T)$ and the following holds. Uniformly on finite time intervals, the solution $u^{\varepsilon}$ to the Benjamin-Ono equation~\eqref{eq:bo} with parameter $\varepsilon$ and initial data $u_0^{\varepsilon}$ converges weakly to $u_{alt}^B$ in $L^2_{r,0}(\T)$ as $\varepsilon\to 0$:
\[
u^{\varepsilon}\weaklim{}{} u^B_{alt}.
\]
If $u_0(x)=-\beta\cos(x)$ for some $\beta>0$, one can choose $u_0^{\varepsilon}=u_0$ for every $\varepsilon>0$.
\end{thm}

\begin{rk}[Strong (resp. weak) convergence before  (reps. after) the shock time]
As long as $u^B$ is a well-defined function, one knows thanks to the conservation of the $L^2$ norm that the convergence in Theorem~\ref{thm:main} is strong.

However, for instance when $u_0(x)=-\beta\cos(x)$, the convergence cannot be strong right after the breaking time $T$ for the Burgers equation. Indeed, for $(t-T)$ positive and small enough, there holds (see Lemma~\ref{lem:CV_weak}) 
\[
\|u^{\varepsilon}(t)\|_{L^2(\T)}=\|u_0\|_{L^2(\T)}>\|u^B_{alt}\|_{L^2(\T)}.
\]
\end{rk}

\begin{rk}[Convergence for small times]
For $\classeC^{\infty}$ initial data, a WKB approximation of the form
\[
u^{\varepsilon}(t,x)=\sum_{j=0}^{+\infty}a_j(t,x)\varepsilon^j
\]
would enable us to get an asymptotic expansion for the zero dispersion limit up to the shock time, by transforming the problem into the Burgers equation for $a_0$ and transport equations for the higher order terms. However, this approach would not give access to information on the solution after the shock formation.
\end{rk}

%



\paragraph{Zero-dispersion limit for the Benjamin-Ono equation on the torus}

To the best of our knowledge, not much seems to be known concerning the zero-dispersion limit for the Benjamin-Ono equation on the torus. A first approach using Whitham modulation approximation can be found in the work of Matsuno~\cite{Matsuno1998}. More recently, Moll~\cite{Moll2019-1} gives a convergence result of the Lax eigenvalues in the zero-dispersion limit. However, the type of convergence is not sufficient to establish that the corresponding approximate solution is a good approximation in the classical space $L^2(\T)$. In this direction, one can also mention a similar approach in~\cite{Moll2019-2} for the quantum periodic Benjamin-Ono equation.

Our strategy relies on the recent work of Gérard, Kappeler and Topalov who constructed  Birkhoff coordinates for the Benjamin-Ono equation in~\cite{GerardKappeler2019}, which adapt to equation~\eqref{eq:bo} through a rescaling. A further study of this transformation appears in~\cite{GerardKappelerTopalov2020, GerardKappelerTopalov2020-2, GerardKappelerTopalov2021, GerardKappelerTopalov2021analyticity} (see also~\cite{Gerard2019} for a survey on the topic).

\paragraph{Zero-dispersion limit for the Benjamin-Ono equation on the real line}

Theorem~\ref{thm:main} is similar to Benjamin-Ono equation the real line studied by Miller and Xu~\cite{MillerXu2011}, where the  authors prove the following result. Let $u_0$ be an initial data satisfying some admissibility conditions, let $u_0^B(t,x)<\dots<u_{2P(t,x)}^B(t,x)$ be the branches of the multivalued solution to the inviscid Burgers equation obtained by the method of characteristics, and let $u^B_{alt}$ be the signed sum of branches as in~\eqref{eq:u_alt}. Then there exists a sequence of initial data $u_0^{\varepsilon}$ such that uniformly on compact time intervals, the solution $u^{\varepsilon}$ is weakly convergent  in $L^2(\R)$ to $u^B_{alt}$. This result also implies strong convergence for all $0\leq t<T$, where $T$ is the breaking time for the Burgers equation.
However, the approach for the zero-dispersion limit for the Benjamin-Ono on the line needs to be restricted to positive initial data with prescribed tail behavior 
\[
|x|^{q+1}\partial_x u_0\longrightarroww{}{x\to\pm\infty}C_{\pm}
\] for some $q>1$, and a generalization to more general potentials as in Theorem~\ref{thm:main} seems still unknown. Finally, Miller and Xu establish a similar zero-dispersion convergence result for the Benjamin-Ono hierarchy in~\cite{MillerXu2011hierarchy}, and it would be interesting to compare their result to the Benjamin-Ono hierarchy on the torus, see~Remark~\ref{rk:hierarchy}.

The approach of Miller and Xu is based on inverse scattering transform techniques, first formally derived by Fokas and Ablowitz~\cite{AblowitzFokas1983}, and then rigorously written by Coifman and Wickerhauser~\cite{CoifmanWickerhauser1990} for small and decaying data, see also~\cite{KaupMatsuno1998}. The strategy is as follows. The initial data is first approximated by a rational potential of Klaus-Shaw type ${u^{\varepsilon}}_0$, that is, a rational potential with only one bump. The authors first guess the right  approximate eigenvalues ${\lambda_n}(\varepsilon)$ and phase constants of $u^{\varepsilon}(0)$ in order for the scattering problem to approximate well the solution. Then they prove that for every time $t$ (in particular for $t=0$), there holds weak convergence of $ {u^{\varepsilon}}(t)$ to $u^B_{alt}(t)$ as $\varepsilon\to 0$. This approximation is necessary in order to have an explicit inversion formula for the scattering data.  A possible progress might come from the recent work on the direct scattering problem for the Benjamin-Ono equation~\cite{Wu2016,Wu2017}, and from the construction of a Birkhoff map started in the paper of Sun~\cite{Sun2020}.

In~\cite{MillerWetzel2016rational}, Miller and Wetzel establish exact formulae for positive rational initial conditions with simple poles.  Using these formulae, the authors are able to derive a precise asymptotic expansion for the scattering data in the zero dispersion limit~\cite{MillerWetzel2016}. 
 In this special case, the asymptotics enable to choose the initial data $u_0$ itself instead of an $\varepsilon$-dependent initial data $u^{\varepsilon}(0)$ in the zero-dispersion limit problem. As we will see in part~\ref{part:inversion_formula}, however, this approach seems uncertain for rational potentials on the torus. On the torus, we are still able to provide a precise asymptotic expansion for the eigenvalues when $u_0(x)=-\beta\cos(x)$ is not a finite gap potential in Theorem~\ref{thm:asymptotics} below, and we hope to extend this asymptotic expansion to more general initial data in a subsequent work.

\paragraph{Zero-dispersion limit for the KdV equation}

The zero-dispersion limit problem was first investigated for the Korteveg-de Vries equation on the real line by Lax and Levermore~\cite{LaxLevermore}
\[
\partial_t u-3\partial_x(u^2)+\varepsilon^2\partial_{xxx}u=0,
\]
when the initial data is negative or zero and decays sufficiently fast at infinity. 
 The authors construct approximate scattering data, or approximate initial data $u_0^{\varepsilon}$, such that the solutions $u^{\varepsilon}(t)$ are convergent to some limit in the weak sense when $\varepsilon\to 0$, uniformly on compact time intervals. The limit is different from the Benjamin-Ono equation, and can only be expressed implicitly as the solution of some variational problem. This approach was adapted to positive initial data in~\cite{Venakides1985}. 
The authors use WKB methods in order to approximate the scattering data associated to the KdV equation, and their analysis is based on the inverse scattering transform. Venakides then describes the nature of oscillations that appear after the dispersive shock time in~\cite{Venakides1991}. The theory was developed in~\cite{DeiftVenakidesZhou1997} using the steepest-descent method in order to get strong convergence results. A further refinement of these asymptotics can be found in the work of Claeys and Grava \cite{ClaeysGrava2009universality, ClaeysGrava2010painleve, ClaeysGrava2010solitonic}, who exhibit in particular a universal wave profile starting from $\varepsilon$-independent initial data.

On the torus, Venakides~\cite{Venakides1987} computes the weak zero dispersion limit for periodic initial data. The author proves that a shock appears for small dispersion parameter at the breaking time of the Burgers equation, causing the emergence of rapid oscillations with wavenumbers and frequencies of order $\gdO(1/\varepsilon)$.  In this purpose, he establishes asymptotics on the exact solution instead of relying on an approximation. More recently, an asymptotic expansion of the spectral parameters has then been established in~\cite{DengBiondiniTrillo2016} with the cosine initial data in order to justify the Zabusky-Kruskal experiment~\cite{ZabuskyKruskal1965}.


\subsection{Distribution of the Lax eigenvalues}

Our strategy of proof relies on the complete integrability of the Benjamin-Ono equation in the sense that it admits Birkhoff coordinates, constructed in in~\cite{GerardKappeler2019}. This transformation enables us to consider general initial data in $L^2_{r,0}(\T)$, that is, in $L^2(\T)$, real-valued and with mean zero. The construction of Birkhoff coordinates relies on the eigenvalues $\lambda_n(u_0;\varepsilon)$ of the Lax operator $L_{u_0}(\varepsilon)$, and some phase constants $\theta_n(u_0;\varepsilon)$ depending on the eigenfunction of $L_{u_0}(\varepsilon)$ (see part~\ref{part:notation_lax} for more details). A careful study of the spectral parameters $\lambda_n(u_0;\varepsilon)$ and $\theta_n(u_0;\varepsilon)$ for $n\geq 1$ leads us to introduce the following asymptotic approximation for the Lax eigenvalues and phase constants.


Let us denote for $\eta\in\R$
\[
F(\eta):=\frac{1}{2\pi}\operatorname{Leb}(x\in[0,2\pi]\mid u_0(x)\geq\eta).
\] 
The eigenvalues are expected to follow some quantization rule depending on the distribution function $F$ as follows (see also Figure~\ref{fig:distribution}).

\begin{mydef}[Admissible approximate initial data]\label{def:approximate_Birkhoff}
Let $u_0$ be a single well initial data. We say that the family of approximate initial data $(u_0^{\varepsilon})_{\varepsilon>0}$  is admissible if $\|u_0^{\varepsilon}\|_{L^2(\T)}\to\|u_0\|_{L^2(\T)}$ as $\varepsilon\to 0$ and the following holds. For every $\delta>0$, there exist $C(\delta), K(\delta)>0$ such that for every $\varepsilon>0$, the eigenvalues $\lambda_n=\lambda_n(u_0^\varepsilon;\varepsilon)$ and phase constants $\theta_n=\theta_n(u_0^{\varepsilon};\varepsilon)$ satisfy the following properties. 
\begin{enumerate}
\item (Small eigenvalues) If $\lambda_n+\varepsilon,\lambda_p+\varepsilon\in\Lambda_-(\delta)=[-\max(u_0)+\delta,-\min(u_0)-\delta]$, then for $\varepsilon\leq \varepsilon_0(\delta)$ small enough,
\begin{equation*}
\left|\int_{-\lambda_p}^{-\lambda_n}F(\eta)\d\eta-(p-n)\varepsilon\right|
	\leq C(\delta)\varepsilon\sqrt{\varepsilon}.
\end{equation*}
\item (Large eigenvalues) We denote $\Lambda_+(\delta)=[-\min(u_0)+\delta,\infty)$. If $\lambda_n+\varepsilon,\lambda_p+\varepsilon\in\Lambda_+(\delta)=[-\min(u_0)+\delta,K(\delta)]$, then 
\[
|\lambda_p-\lambda_n-(p-n)\varepsilon|\leq C(\delta)\varepsilon\sqrt{\varepsilon},
\]
whereas if $\lambda_n+\varepsilon,\lambda_p+\varepsilon\in [K(\delta),\infty)$, then 
\[
|\lambda_p-\lambda_n-(p-n)\varepsilon|\leq \delta.
\]
\item (Other eigenvalues) There are  at least $\frac{\delta}{C\varepsilon}$ and at most $\frac{C\delta}{\varepsilon}$ eigenvalues $\lambda_n$ satisfying $\lambda_n+\varepsilon\in[-\min(u_0)-\delta,-\min(u_0)+\delta]$, and at   least $\frac{1}{C(\delta)\varepsilon}$ and at most $\frac{C\delta}{\varepsilon}$ eigenvalues in the region  $\lambda_n+\varepsilon\in[-\max(u_0),-\max(u_0)+\delta]$.
\item (Phase constants)
\begin{equation}\label{eq:phase}
\left|\theta_{n+1}-\theta_n-\left(\pi-\frac{x_+(-\lambda_n)+x_-(-\lambda_n)}{2}\right)\right|\leq C(\delta)\varepsilon\sqrt{\varepsilon}.
\end{equation}
\end{enumerate}
\end{mydef}

By applying the inverse Birkhoff transformation $(\Phi^{\varepsilon})^{-1}$, the choice of a family of eigenvalues and phase factors defines an approximate initial data $u_0^{\varepsilon}$. The set of admissible initial data is never empty, see Lemma~\ref{lem:exist_admissible}.


This distribution of the spectral parameters is completely justified in the case of the cosine initial data thanks to the following Theorem, which is our second main result.
\begin{thm}[Distribution of the Lax eigenvalues for cosine initial data]\label{thm:asymptotics}
Assume that $u_0(x)=-\beta\cos(x)$ for some $\beta>0$. 
Then the family $(u_0)_{\varepsilon>0}$ of $\varepsilon$-independent initial data is admissible.
\end{thm}

\begin{figure}
\begin{center}
\includegraphics[scale=0.32]{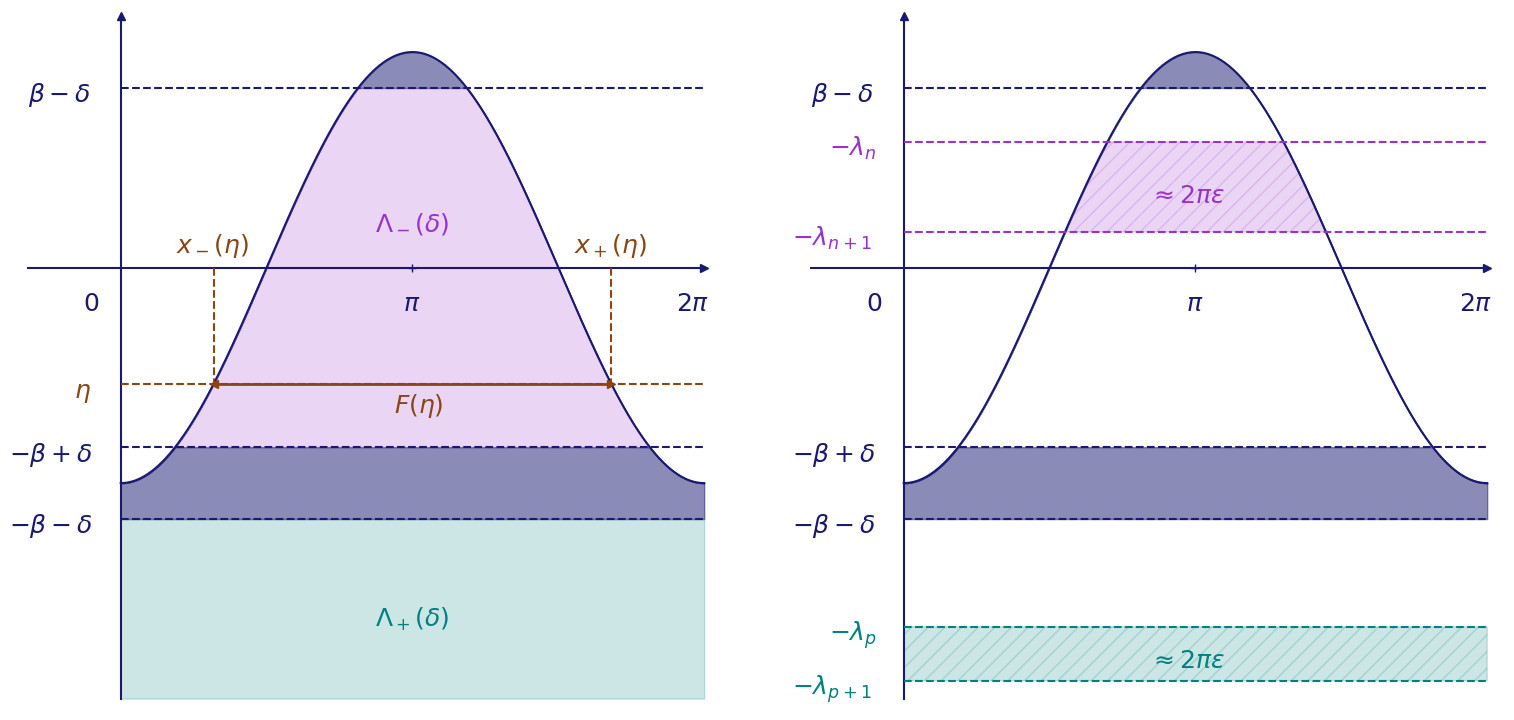}
\end{center}
\caption{Definition of $\Lambda_{\pm}(\delta)$ (on the left) and distribution of the eigenvalues in the zero-dispersion limit (on the right), with initial data $u_0(x)=-\beta\cos(x)$}\label{fig:distribution}
\end{figure}

\begin{rk}[Other eigenvalues]
Note that we need to remove two small regions $\lambda+\varepsilon\in [-\beta,-\beta+\delta)$, $\lambda+\varepsilon\in (\beta-\delta,\beta+\delta)$ and one large region $\lambda+\varepsilon\in(K(\delta),\infty)$ in our analysis, for which a uniform asymptotic expansions of the eigenvalues is not known. The reason is that we need uniform bounds in the method of stationary phase and in the Laplace method, but the stationary point goes to the limits of integration at $-\beta$, $\beta$ and $\infty$.
\end{rk}


To prove Theorem~\ref{thm:asymptotics}, we will  see that the eigenfunctions $f_n(u_0;\varepsilon)\in L^2_+(\T)$ of the Lax operator must have a holomorphic extension to the open unit disk in the complex plane. Therefore, the pole at $0$ has to satisfy some constraints in order to meet the holomorphy property, leading to an asymptotic expansion of the eigenvalues $\lambda_n(u_0;\varepsilon)$ as $\varepsilon\to 0$.

As a consequence, we will see that the initial data $u_0(x)=-\beta\cos(x)$ is well-approximated by a solution for which $\gamma_n(u_0;\varepsilon)=|\zeta_n(u_0;\varepsilon)|^2=0$ when $n\geq\frac{\beta+\gdO(1)}{\varepsilon}$, which is actually to a $N_{\varepsilon}$-soliton with $N_{\varepsilon}\approx\frac{\beta}{\varepsilon}$. This insight leads us to define approximate Lax eigenvalues for more general single well initial data from the distribution function $F$ in Definition~\ref{def:approximate_Birkhoff}. In order to derive the entire Birkhoff coordinates which completely characterize the approximate initial data $u_0^{\varepsilon}$, we needed to choose the phase constants $\theta_n(u_0^{\varepsilon};\varepsilon)$ for $n\geq 1$, but we dot not have a justification for our choice yet. When $u_0(x)=-\beta\cos(x)$, we take advantage of the fact that all the phase constants vanish.

\subsection{Strategy of proof}\label{part:strategy}

Let $u^{\varepsilon}$ be the solution to~\eqref{eq:bo} with initial data $u_0^{\varepsilon}$ satisfying the admissibility conditions from Definition~\ref{def:approximate_Birkhoff}. In this part, we explain how to obtain weak convergence of $u^{\varepsilon}$ to $u^{B}_{alt}$ when $\varepsilon\to 0$.

We first make a link between the matrix $M(u_0;\varepsilon)=(M_{n,p}(u_0;\varepsilon))_{n,p\geq 0}$ of the shift operator $S:h\in L^2_+(\T)\mapsto e^{ix}h\in L^2_+(\T)$ (defined later in~\eqref{def:M}) as a function of the spectral parameters, and the Fourier coefficients of $u_0\in L^2_{r,0}(\T)$ (see Proposition~\ref{prop:u-trM})
\[
\widehat{u_0}(k)=\varepsilon\Tr(M(u_0;\varepsilon)^k).
\]
Then, it is quite direct that for a single well potential there also holds (see Proposition~\ref{prop:formule_uk})
\[
\widehat{u_0}(k)=\frac{-i}{2k\pi}\int_{\min u_0}^{\max u_0}(e^{-i k x_+(\eta)}-e^{-i k x_-(\eta)})\d\eta,
\]
where $x_{\pm}(\lambda)$ are the antecedents of $\lambda$ by $u_0$ defined below Definition~\ref{def:single well}.
One can make a parallel between those two formulas and~\cite{MillerXu2011}, where Miller and Xu make a link between the $k$-th moment $\Tr(M(u_0;\varepsilon)^k)$ and some integral depending on $(x+2\lambda t-x_{\pm}(\lambda))^{k+1}$ (see their Proposition 4.1).%

Using the asymptotics of the eigenvalues, we justify an approximation of $\varepsilon\Tr(M(u_0^{\varepsilon};\varepsilon)^k)$ under the form
\begin{equation}\label{eq:fourier_u_epsilon}
\varepsilon\Tr(M(u_0^{\varepsilon};\varepsilon)^k)\approx  \frac{-i}{2k\pi}\int_{\min u_0}^{\max u_0}(e^{-i k x_+(\eta)}-e^{-i k x_-(\eta)})\d\eta,
\end{equation}
where the right-hand side does not depend on $\varepsilon$ but only on $u_0$. In this approximation, we neglect the terms $M_{n,p}(u_0^{\varepsilon};\varepsilon)$ which are off-diagonal $|n-p|\geq\varepsilon^{-r}$ for some $\varepsilon>0$, and the terms where the index $n$ corresponds to an eigenvalue outside the small-eigenvalue region $\lambda_n+\varepsilon\in\Lambda_-(\delta)$. Our approximation method then enables us to deduce an approximation of $\varepsilon\Tr(M(u^{\varepsilon}(t);\varepsilon)^k)$:
\[
\varepsilon\Tr(M(u^{\varepsilon}(t);\varepsilon)^k)\approx  \frac{-i}{2k\pi}\int_{\min u_0}^{\max u_0}(e^{-i k (x_+(\eta)+2\eta t)}-e^{-i k( x_-(\eta)+2\eta t)})\d\eta.
\]
For small times, we recognize the right hand side as the $k$-th Fourier coefficient for the time evolution of $u_0$ under the Burgers equation.
Indeed, for the Burgers solution $u^B$, one would have $x_{\pm}^B(\eta,t)=x_{\pm}(\eta)+2\eta t$, so that we have actually proven
\[
\widehat{u^{\varepsilon}(t)}(k)
	\approx \widehat{u^B(t)}(k).
\]
 After the breaking time for the Burgers equation, the right hand side becomes the alternate sum of Fourier coefficients for the branches of the multivalued solution of the Burgers equation obtained with the method of characteristics, denoted $u^B_{alt}$.

\subsection{Plan of the paper}

 In section~\ref{section:notation}, we introduce the Birkhoff coordinates and spectral parameters associated to~\eqref{eq:bo}. In section~\ref{part:approximation}, we use the approximate spectral data in order to approximate the $k$-th Fourier coefficient of the solution at time $t$ by the $k$-th Fourier coefficient of the relevant solution to the Burgers equation.  In section~\ref{part:Lax}, we establish an asymptotic expansion of the Lax eigenvalues associated to equation~\eqref{eq:bo} for the cosine initial data.

\paragraph{Acknowledgments} The author is grateful to Patrick Gérard for relevant advice on this problem, in particular, for providing notes on Proposition~\ref{prop:eigenvalue_equation} about the cosine function. This material is based upon work supported by the National Science Foundation under Grant No. DMS-1439786 while the author was in residence at the Institute for Computational and Experimental Research in Mathematics in Providence, RI, during the ``Hamiltonian Methods in Dispersive and Wave Evolution Equations'' program.

\section{Birkhoff coordinates in the zero dispersion limit}\label{section:notation}

In all that follows, the constants $C$ may change from line to line, and are denoted by $C(\delta)$ if they depend on a parameter $\delta$ which is not fixed. Some parameters $0<c<r<1$ are also fixed all throughout this paper.

\subsection{Lax eigenvalues for the Benjamin-Ono equation}\label{part:notation_lax}

Our main tool is the description of complete integrability for the Benjamin-Ono equation on the torus from~\cite{GerardKappeler2019}, in which Birkhoff coordinates are constructed in the case $\varepsilon=1$. Let us adapt the setting to equation~\eqref{eq:bo} with arbitrary parameter $\varepsilon$.

For $u\in L^2_{r,0}(\T)$, we denote $L_u(\varepsilon)$ the Lax operator for the Benjamin-Ono equation with parameter~$\varepsilon$ 
\[
L_u(\varepsilon)=\varepsilon D-T_u.
\]
We have written $D=-i\partial_x$, moreover, $T_u:h\in L^2_+(\T)\mapsto \Pi(uh)\in L^2_+(\T)$ is a Toeplitz operator, where $\Pi$ is the Szeg\H{o} projector onto the subspace $L^2_+(\T)$ of $L^2(\T)$ of functions with positive Fourier frequencies.

Let $(\lambda_n(u;\varepsilon))_{n\geq 0}$ be the eigenvalues of $L_u(\varepsilon)$ in increasing order. Let $(f_n(u;\varepsilon))_{n\geq 0}$ be the corresponding eigenfunctions defined through the additional conditions 
\(
\langle \un|f_0\rangle>0
\)
{ and }
\(\langle f_n|e^{ix}f_{n-1}\rangle>0\), \(n\geq 1.
\)

Then according to~\cite{GerardKappeler2019}, the gaps lengths
\[
\gamma_n(u;\varepsilon)=\lambda_n(u;\varepsilon)-\lambda_{n-1}(u;\varepsilon)-\varepsilon
\]
are nonnegative, and for any $n\geq 1$, there holds
\begin{equation*}
\lambda_n(u;\varepsilon)
	=n\varepsilon+\lambda_0(u;\varepsilon)+\sum_{k=1}^n\gamma_k(u;\varepsilon)
	=n\varepsilon-\sum_{k=n+1}^{+\infty}\gamma_k(u;\varepsilon).
\end{equation*}
Finally, defining \(
h^{1/2}=\{(\zeta_n)_{n\geq 1}\in \C^{\N}\mid \sum_{n\geq 1}n|\zeta_n|^2<\infty\},
\)
the Birkhoff transform is written
\[
\Phi^{\varepsilon}:u\in L^2_{r,0}(\T)\mapsto (\zeta_n(u;\varepsilon))_{n\geq 1}\in h^{1/2}.
\]
In the construction of this transformation, one can see that $|\zeta_n(u;\varepsilon)|^2=\gamma_n(u;\varepsilon)$. Moreover, the phase constants are defined by $\theta_n(u;\varepsilon)=\arg(\zeta_n(u;\varepsilon))$ for $\zeta_n(u;\varepsilon)\neq 0$.

The following Parseval formula holds true
\begin{equation}\label{eq:Parseval}
\|u\|_{L^2(\T)}^2=2\varepsilon\sum_{n\geq 1}n\gamma_n(u;\varepsilon).
\end{equation}
In order to check this identity for general parameter $\varepsilon$, it might be useful to note that $L_u(\varepsilon)=\varepsilon L_{u/\varepsilon}(1)$, so that $f_n(u;\varepsilon)=f_n(u/\varepsilon;1)$, then
\(
\zeta_n(u;\varepsilon)=\sqrt{\varepsilon}\zeta_n(u/\varepsilon;1),
\)
and finally $\gamma_k(u;\varepsilon)=\varepsilon\gamma_k(u/{\varepsilon};1)$. As a consequence, we can use the Parseval formula with parameter $\varepsilon=1$ from~\cite{GerardKappeler2019}. 


Finally, the eigenvalue $\lambda_n(u;\varepsilon)$ satisfies the min-max formula
\[
\lambda_n(u;\varepsilon)=\max_{\dim F=n} \min\{\langle L_u(\varepsilon)h|h\rangle\mid h\in H^1_+\cap F^{\perp},\;\|h\|_{L^2}=1\}.
\]
Therefore, the smallest eigenvalue
\begin{align*}
\lambda_0(u;\varepsilon)
	&=\min\{\langle L_u(\varepsilon)h|h\rangle\mid h\in H^1_+,\;\|h\|_{L^2}=1\}
\end{align*}
is bounded below by
\begin{equation}\label{eq:lambda0}
\lambda_0(u;0)
	=-\max\{\langle u ||h|^2\rangle\mid h\in H^1_+,\;\|h\|_{L^2}=1\}
	\geq -\max_{x\in\T}u(x).
\end{equation}
\subsection{Inversion formula}\label{part:inversion_formula}

Given $u\in L^2_{r,0}(\T)$, one has the inversion formula~\cite{GerardKappeler2019}
\begin{equation}\label{eq:inversionM}
\Pi u(z)=\left\langle (\operatorname{Id}-zM(u;\varepsilon))^{-1}X(u;\varepsilon)|Y(u;\varepsilon)\right\rangle
\end{equation}
with
\[
X(u;\varepsilon)=(-\lambda_p(u;\varepsilon)\langle \un|f_p(u;\varepsilon)\rangle)_{p\geq 0},
\quad
Y(u;\varepsilon)=(\langle \un|f_n(u;\varepsilon)\rangle)_{n\geq 0},
\]
and $M(u;\varepsilon)$ is the matrix of the shift operator $S:u\in L^2_+\mapsto e^{ix}u\in L^2_+$ with coefficients
\[M_{n,p}(u;\varepsilon)=\langle f_p(u;\varepsilon)|Sf_n(u;\varepsilon)\rangle,
\]
\begin{equation}\label{def:M}
M_{n,p}(u;\varepsilon)
	=\begin{cases}
\sqrt{\mu_{n+1}(u;\varepsilon)}\frac{\sqrt{\kappa_p(u;\varepsilon)}}{\sqrt{\kappa_{n+1}(u;\varepsilon)}}\overline{\zeta_p(u;\varepsilon)}\zeta_{n+1}(u;\varepsilon)\frac{1}{\lambda_p(u;\varepsilon)-\lambda_n(u;\varepsilon)-\varepsilon}
& \text{ if } p\neq n+1
\\
\sqrt{\mu_{n+1}(u;\varepsilon)}
& \text{ if } p= n+1.
\end{cases}
\end{equation}
In this definition, $\mu_n$ and $\kappa_n$ are functions of the eigenvalues $(\lambda_n(u;\varepsilon))_n$ defined in~\eqref{eq:kappa_n} and in~\eqref{eq:mu_n}.

In order to get an asymptotic expansion of the eigenvalues, the strategy used in Miller, Wetzel~\cite{MillerWetzel2016,MillerWetzel2016rational}    is restricted to $N$-gap solutions. On the torus, when $u$ satisfies $\zeta_n(u;\varepsilon)=0$ for every $n\geq N$ ($u$ is a $N$-gap for the parameter $\varepsilon$), the inversion formula becomes
\[
\Pi u(z)=-2\varepsilon z\partial_z   \log  \det \left(\operatorname{Id} -zM_{N-1}(u;\varepsilon)\right).
\]
Recall that a function in the Hardy space $f\in L^2_+(\T)$ admits an holomorphic expansion to the open unit disk $\D$ in $\C$ as follows. Expand $f$ in Fourier modes
\(
f(x)=\sum_{n\geq 0}c_ne^{inx},
\)
then the holomorphic expansion of $f$ is written
\(
f(z)=\sum_{n\geq 0}c_nz^n.
\)
Conversely, let $z\in\overline{\mathbb{D}}\subset\C$ and $Q(z)=\prod_{j=1}^{N}(1-\overline{q_j}z)$ with $q_j\in\C$ and $0<|q_j|<1$, then
\[
\Pi u(z)=\varepsilon\sum_{j=1}^{N}\frac{\overline{q_j}z}{1-\overline{q_j}z}
\]
defines a $N$-gap for the parameter $\varepsilon$ and $Q(z)=\det (\operatorname{Id} -zM_{N-1}(u;\varepsilon))$.
Such a $N$-gap $u$ has a meromorphic extension on $\C$
\[
u(z)=\varepsilon\sum_{j=1}^{N}\frac{\overline{q_j}z}{1-\overline{q_j}z}+\varepsilon\sum_{j=1}^{N}\frac{q_j}{z-q_j},
\]
with poles $q_j$ inside the unit disk and $1/\overline{q_j}$ outside of the unit disk. One difficulty of this approach on the torus is the following observation. If we replace $\varepsilon$ by $\varepsilon/2$, then
\[
\Pi u(z)=-\frac{\varepsilon}{2}\sum_{j=0}^{2N-1}\frac{q_jz}{1-q_jz}
\]
with $q_{j+N}=q_j$, and $u$ becomes a $2N$-gap for the parameter $\varepsilon/2$. As a consequence, we expect that a $N$-gap for the parameter $\varepsilon$ becomes a $N/\varepsilon$-gap for the parameter $\varepsilon$. The number of poles and zeroes of the meromorphic extension of $u$ being increasing with $\varepsilon$, we do not expect to get a uniform approximation of the eigenvalues when using the steepest descent method as $\varepsilon\to 0$.


Instead of using this approach, we rather consider the trigonometric polynomial $-\beta\cos$, which meromorphic extension to the complex plane $\frac{z+z^{-1}}{2}$ has a pole at the origin $z=0$ but nowhere else. This pole could have a higher order if one considers more general trigonometric polynomials, but this order does not depend on $\varepsilon$, which makes us hope to be able to extend our result to trigonometric polynomials in the future.

\subsection{Approximate Birkhoff coordinates}
In this part, we justify that if Theorem~\ref{thm:asymptotics} is true for $u_0(x)=-\beta\cos(x)$, then $(u_0)_{\varepsilon>0}$ defines an admissible family of initial data in the sense of Definition~\ref{def:approximate_Birkhoff}. Then we state some consequences of Definition~\ref{def:approximate_Birkhoff}.

Concerning the cosine function, it is enough to establish that all phase constants are  equal to $0$. 

\begin{prop}[Phase constants for the cosine function]\label{prop:cos} Fix a parameter $\beta>0$ and consider the initial data $u(x)=-\beta\cos(x)$. For every $n\geq 1$, there holds $\zeta_n(u;\varepsilon)>0$, hence $\theta_n(u;\varepsilon)=0$.
\end{prop}

\begin{proof}
It is enough to tackle the case $\varepsilon=1$. From~\cite{GerardKappeler2019}, equation~(2.7), one has for every $n\geq 0$
\[
L_{u}Sf_n=SL_{u}f_n+Sf_n-\langle uSf_n|\un\rangle\un.
\]
For the potential $u(x)=-\beta\cos(x)$, since $Sf_n\in L^2_+(\T)$, there holds 
\[\langle u|Sf_n\rangle=-\frac{\beta}{2}\langle e^{ix}|Sf_n\rangle=-\frac{\beta}{2}\langle \un|f_n\rangle,
\]
so that
\[
L_{u}Sf_n=(\lambda_n+1)Sf_n+\frac{\beta}{2}\langle f_n|\un\rangle\un.
\]
Note that $\gamma_n=\kappa_n|\langle \un|f_n\rangle|^2$, where $\kappa_n$ is nonzero and defined in~\eqref{eq:kappa_n}.
If $\gamma_n=0$ for some $n\geq 1$, then $L_{u}Sf_n=(\lambda_n+1)Sf_n$, which implies that $f_{n+1}=Sf_n$, $\lambda_{n+1}=\lambda_n+1$ and $\gamma_{n+1}=0$. Conversely, if $\gamma_{n+1}=0$, then Lemma 2.5 from~\cite{GerardKappeler2019} implies that $L_{u}Sf_n=\lambda_{n+1}Sf_n=(\lambda_n+1)Sf_n$, consequently, $\langle \un|f_n\rangle=0$ and $\gamma_n=0$. We conclude that either all the Birkhoff coordinates of $u$ are zero, which is impossible, either all the Birkhoff coordinates of $u$ are nonzero.

Finally, from~\cite{GerardKappeler2019}, Lemma 2.6, one has for every $n,p\geq 0$
\begin{equation*}\label{eq:GK_fn}
(\lambda_{p}-\lambda_n-1)\langle f_{p}|Sf_n\rangle=-\langle f_{p}|\un\rangle\langle u|Sf_n\rangle.
\end{equation*}
When $p=n+1$, we get
\[
\gamma_{n+1}\langle f_{n+1}|Sf_n\rangle=\frac{\beta}{2}\langle f_{n+1}|\un\rangle \langle \un|f_n\rangle.
\]
Since $\gamma_{n+1}>0$ and $\langle f_{n+1}|Sf_n\rangle>0$ by definition of the eigenfunctions $f_n$, we deduce that for every $n\geq 0$, there holds
\[
\overline{\zeta_{n+1}}\zeta_n>0.
\]
Using that $\zeta_0=1$, we conclude that the Birkhoff coordinates of $u$ are all real and positive.
\end{proof}

The following result summarizes the properties of the approximate Birkhoff coordinates both in the case $u_0(x)=-\beta\cos(x)$ and in the case of a general single well potential.

\begin{cor}[Consequences of Definition~\ref{def:approximate_Birkhoff}]\label{cor:asymptotics}
Let $u_0$ be a single well initial data. We choose an admissible family $(u_0^{\varepsilon})_{\varepsilon>0}$ as in Definition~\ref{def:approximate_Birkhoff}. We consider the approximate Lax eigenvalues $\lambda_n=\lambda_n(u_0^{\varepsilon};\varepsilon)$. 
 Let $\delta>0$, then for $\varepsilon\leq \varepsilon_0(\delta)$ small enough the following holds.
\begin{enumerate}
\item For every $n\geq 0$, we have
\begin{equation}\label{eq:asymptotics}
\left|\int_{-\lambda_n}^{\max(u_0)}F(\eta)\d\eta-n\varepsilon\right|
	\leq C\delta+C(\delta)\varepsilon\sqrt{\varepsilon}.
\end{equation}
\item (Large eigenvalues) If $-\lambda_n-\varepsilon\in\Lambda_-(\delta)=[-\max(u_0)+\delta,-\min(u_0)-\delta]$, then there holds
\[
\sum_{k= n+1}^{\infty}\gamma_k\leq C\delta+ C(\delta)\varepsilon\sqrt{\varepsilon}.
\]

\item (Two-region eigenvalues) 
If $-\lambda_n-\varepsilon\in\Lambda_-(\delta)$ and $-\lambda_p-\varepsilon\in\Lambda_+(\delta)=[-\min(u_0)+\delta,\infty)$, then $|p-n|\geq \frac{\delta}{C\varepsilon}$.
\end{enumerate}
\end{cor}

We also state the bounds that we will use on the distribution function.
\begin{lem}[Lipschitz properties of the distribution function]\label{lem:F} Fix a general single well initial data. Let $\delta>0$. There exists $C(\delta)>0$ such that $F$, $x_+$ and $x_-$ are $C(\delta)$-Lipschitz on $[\min(u_0)+\delta,\max(u_0)-\delta]$. Moreover, $F(\eta)\geq 1/C(\delta)$ for $\eta\leq \max(u_0)-\delta$. 
\end{lem}

\section{Asymptotic expansion of the Fourier coefficients for single well initial data}\label{part:approximation}

In this section, we establish Theorem~\ref{thm:main} by proving the convergence of every Fourier coefficient of $u^{\varepsilon}(t)$ to the Fourier coefficient of $u_{alt}^B(t)$.

\subsection{Fourier coefficients as a trace}

\begin{prop}[Fourier coefficients and trace of the shift matrix]\label{prop:u-trM}
For any $u\in L^2_{r,0}(\T)$ and $k\geq 1$, there holds
\[
\widehat{u}(k)
	=\varepsilon \Tr(M(u;\varepsilon)^k).
\]
\end{prop}
\begin{proof}
Let $\hamilton_1(u)=\frac{1}{2\pi}\int_{0}^{2\pi}u^2(x)\d x$ be the mass of the solution. We use the differentiation formula 
\[\d \hamilton_1(u).\cos(kx)-i\d \hamilton_1(u).\sin(kx)=2\langle u|\cos(kx)\rangle-2i \langle u|\sin(kx)\rangle=2\widehat{u}(k).
\]
Since the Parseval formula gives
\[
\hamilton_1(u)
	=2\varepsilon\sum_{n\geq 1}n\gamma_n(u;\varepsilon),
\]
taking the differential leads to
\[
\d\hamilton_1(u).h=2\varepsilon\sum_{n\geq 1}n\d\gamma_n(u;\varepsilon).h.
\]
Given that $\gamma_n(u;\varepsilon)=\varepsilon\gamma_n(u/\varepsilon;1)$ and $f_n(u;\varepsilon)=f_n(u/\varepsilon;1)$, one has 
\[
\d\gamma_n(u;\varepsilon).h=-\langle|f_n(u;\varepsilon)|^2-|f_{n-1}(u;\varepsilon)|^2|h\rangle.
\]
We get that
\[
\d\hamilton_1(u).h=-2\varepsilon\sum_{n\geq 1}n\langle|f_n(u;\varepsilon)|^2-|f_{n-1}(u;\varepsilon)|^2|h\rangle.
\]
This leads to the telescopic sum
\[
\d\hamilton_1(u).h=2\varepsilon\sum_{n\geq 1}\langle|f_n(u;\varepsilon)|^2|h\rangle.
\]
But we note that
\begin{align*}
\Tr(M(u;\varepsilon)^k)
	&=\sum_n\langle f_n(u;\varepsilon)|S^kf_n(u;\varepsilon)\rangle\\
	&=\sum_n\langle f_n(u;\varepsilon)|\cos(kx) f_n(u;\varepsilon)\rangle-i \sum_n\langle f_n(u;\varepsilon)|\sin(kx) f_n(u;\varepsilon)\rangle,
\end{align*}
which leads to the identity with $h=\cos(kx)$ and $h=\sin(kx)$.
\end{proof}

\begin{rk}
When $\varepsilon$ is fixed, it is possible to prove that the following expanded formula for the trace
\[
\widehat{u}(k)=\varepsilon\Tr(M^k)
	=\varepsilon\sum_{\substack{n_1,\dots,n_{k+1}\geq 1\\ n_1=n_{k+1}} }\prod_{i=1}^{k}M_{n_i,n_{i+1}}.
\]
is absolutely convergent, but we will see in Proposition~\ref{prop:M_ACV} that one can bound the sum of absolute values of the terms by some constant $C(\delta)$ independent of $\varepsilon$.
\end{rk}

In~\cite{MillerXu2011}  Proposition 4.1, Miller and Xu make a link between the $k$-th moment $\Tr(M(u;\varepsilon)^k)$ and some integral depending on $(x+2\lambda t-x_{\pm}(\lambda))^{k+1}$. In our setting, we can guess what is the corresponding formula on the torus by expressing the $k$-th Fourier coefficient of $u$ in a different manner.

\begin{prop}\label{prop:formule_uk}
For any single well potential $u\in \classeC^1_{r,0}(\T)$ (see Definition~\ref{def:single well}), we have
\[
\widehat{u}(k)=\frac{-i}{2k\pi}\int_{\min u}^{\max u}(e^{-i k x_+(\eta)}-e^{-i k x_-(\eta)})\d\eta.
\]
\end{prop}

\begin{proof}
We integrate by parts
\[
2\pi\widehat{u}(k)
	=\left[\frac{u(x)e^{-ikx}}{-ik}\right]_{0}^{2\pi}-\frac{i}{k}\int_{0}^{2\pi}\partial_x u(x) e^{-ikx}\d x,
\]
where the crochet vanishes by periodicity. Let $x_{\max}\in [0,2\pi]$ the unique point for which $u(x_{\max})=\max_{\T}u$. We split the integral between the zones $[0,x_{\max}]$ on which $u$ is increasing, and $[x_{\max},2\pi]$ on which $u$ is decreasing. This leads to
\[
2\pi\widehat{u}(k)
	=-\frac{i}{k}\int_{0}^{x_{\max}}\partial_x u(x) e^{-ikx}\d x-\frac{i}{k}\int_{x_{\max}}^{2\pi}\partial_x u(x) e^{-ikx}\d x.
\]
Then we make the change of variable $\eta=u(x)$ (or $x_-(\eta)=x$) in the first term of the right hand side, and $\eta=u(x)$ (or $x_+(\eta)=u(x)$) in the second term of the right hand side. Since in both cases there holds $\d\eta=\partial_x u(x)\d x$, we get
\[
2\pi\widehat{u}(k)
	=-\frac{i}{k}\int_{\min(u)}^{\max(u)} e^{-ikx_+(\eta)}\d \eta-\frac{i}{k}\int_{\max(u)}^{\min(u)} e^{-ikx_-(\eta)}\d \eta.\qedhere
\]
\end{proof}

\subsection{Upper bounds for the Fourier coefficients}


We now fix some $k\in\Z$ and estimate the $k$-th Fourier coefficient  $\widehat{u_0^{\varepsilon}}(k)$ of the approximate initial data $u_0^{\varepsilon}$, where the rate of convergence may depend on $k$. 

In this part, we first establish some upper bounds on the matrix coefficients $M_{n,p}(u_0^{\varepsilon};\varepsilon)$. As a consequence, we justify that in the formula for $\varepsilon\Tr(M(u_0^{\varepsilon}; \varepsilon)^k)$, we can neglect the terms when the indexes $n,p$ are too off-diagonal $|p-n|\geq \varepsilon^{-r}$ for some $0<r<1$, and when the Lax eigenvalue $\lambda_n$ is not in the region $\lambda_n+\varepsilon\in\Lambda_-(\delta)$.

Up to replacing $u_0^{\varepsilon}$ by some very close initial data $v_0^{\varepsilon}$ in the proof, one can assume that for every $n$, there holds $\gamma_n(u_0^{\varepsilon};\varepsilon)\neq 0$. Indeed, let $T,\delta_0>0$. By continuity of the flow map for~\eqref{eq:bo}, there exists $c_1(\varepsilon)$ such that if $\|u_0^{\varepsilon}-v_0^{\varepsilon}\|_{L^2}\leq c_1(\varepsilon)$, then we have
\begin{equation}\label{ineq:flow_epsilon}
\sup_{t\in[0,T]}\|u^{\varepsilon}(t)-v^{\varepsilon}(t)\|_{L^2}\leq \delta_0.
\end{equation}
By continuity of the inverse Birkhoff map $\Phi(\varepsilon)^{-1}$, there exists $c_2(\varepsilon)>0$ such that if
\[
\sum_{k}\varepsilon k|\zeta_k(u_0^{\varepsilon};\varepsilon)-\zeta_k(v_0^{\varepsilon};\varepsilon)|^2<c_2(\varepsilon),
\]
then $\|u_0^{\varepsilon}-v_0^{\varepsilon}\|_{L^2}\leq c_1(\varepsilon)$. We choose $v_0^{\varepsilon}$ under the form $\zeta_k(v_0^{\varepsilon};\varepsilon)=\varepsilon_k(\varepsilon)>0$ as soon as $\zeta_k(u_0^{\varepsilon};\varepsilon)=0$, $\zeta_k(v_0^{\varepsilon};\varepsilon)=\zeta_k(u_0^{\varepsilon};\varepsilon)$ otherwise, with $\varepsilon_k(\varepsilon)$ small enough so as to satisfy the above inequality. Then inequality~\eqref{ineq:flow_epsilon} holds. As a consequence, for every $t\in [0,T]$, there holds
\[
\sup_{t\in[0,T]}|\widehat{u^{\varepsilon}(t)}(k)-\widehat{v^{\varepsilon}(t)}(k)|\leq \delta_0,
\]
and since $\delta_0$ is arbitrary, the convergence of the Fourier coefficients for $v^{\varepsilon}$ are enough to conclude the proof of Theorem~\ref{thm:main} for the family $(u_0^{\varepsilon})_{\varepsilon}$.

In what follows, we fix $\varepsilon>0$ and drop the $\varepsilon$ in the notation, for instance $\lambda_n$  stands for $\lambda_n(u_0^{\varepsilon};\varepsilon)$. Recall that when $\gamma_{n+1}\neq 0$, then~\cite{GerardKappeler2019} 
\[
M_{n,p}(u_0^{\varepsilon};\varepsilon)=\sqrt{a_n\gamma_{n+1}\gamma_p}\frac{1}{\lambda_p-\lambda_n-\varepsilon}e^{i(\theta_{n+1}-\theta_p)}
\]
where
\[
a_n=\mu_{n+1}\frac{\kappa_p}{\kappa_{n+1}}> 0,
\]
\begin{equation}\label{eq:kappa_n}
\kappa_n=\frac{1}{\lambda_n-\lambda_0}\prod_{\substack{p=1\\p\neq n}}^{\infty}\left(1-\frac{\gamma_p}{\lambda_p-\lambda_n}\right),
\end{equation}
\begin{equation}\label{eq:mu_n}
\mu_{n+1}=\left(1-\frac{\gamma_{n+1}}{\lambda_{n+1}-\lambda_0}\right)\prod_{\substack{p=1\\p\neq n+1}}^{\infty}\frac{\left(1-\frac{\gamma_p}{\lambda_p-\lambda_{n+1}}\right)}{\left(1-\frac{\gamma_p}{\lambda_p-\lambda_n-\varepsilon}\right)}.
\end{equation}

\begin{lem}[Formula for $a_n$] The following formula holds for every $n\geq 1$
\[
a_n\frac{\gamma_n\gamma_{n+1}}{\varepsilon^2}
	=\left(1+\frac{\varepsilon}{\lambda_n-\lambda_0}\right)
	\prod_{\substack{p=1\\p\neq n}}^{\infty} \left(1-\frac{\varepsilon^2}{(\lambda_p-\lambda_n)^2}\right).
\]
\end{lem}
\begin{proof}

We simplify the product
\(
a_n=\mu_{n+1}\frac{\kappa_n}{\kappa_{n+1}}
\)
as
\begin{align*}
a_n
	&=\frac{\lambda_n-\lambda_0+\varepsilon}{\lambda_n-\lambda_0}
	\left(\prod_{\substack{p=1\\p\neq n,n+1}}^{\infty} \frac{\left(1-\frac{\gamma_p}{\lambda_p-\lambda_n}\right)}{\left(1-\frac{\gamma_p}{\lambda_p-\lambda_n-\varepsilon}\right)}\right)
	\frac{1-\frac{\gamma_{n+1}}{\lambda_{n+1}-\lambda_n}}{1-\frac{\gamma_n}{\lambda_n-\lambda_{n+1}}}\cdot
	\frac{1-\frac{\gamma_n}{\lambda_n-\lambda_{n+1}}}{1-\frac{\gamma_n}{-\varepsilon}}\\
	&=\frac{\lambda_n-\lambda_0+\varepsilon}{\lambda_n-\lambda_0}
	\left(\prod_{\substack{p=1\\p\neq n,n+1}}^{\infty} \frac{\lambda_{p-1}-\lambda_n+\varepsilon}{\lambda_p-\lambda_n}\frac{\lambda_p-\lambda_n-\varepsilon}{\lambda_{p-1}-\lambda_n}\right)\frac{\varepsilon^2}{(\varepsilon+\gamma_{n+1})(\varepsilon+\gamma_n)}.
\end{align*}
Then one can re-index the product to get
\begin{multline*}
a_n
	=\frac{\lambda_n-\lambda_0+\varepsilon}{\lambda_n-\lambda_0}
	\left(\prod_{\substack{p=1\\p\neq n}}^{\infty} \left(1+\frac{\varepsilon}{\lambda_p-\lambda_n}\right)\left(1-\frac{\varepsilon}{\lambda_p-\lambda_n}\right)\right)\\
	\left(\frac{\lambda_{n-1}-\lambda_n+\varepsilon}{\lambda_{n-1}-\lambda_n}\right)^{-1}\left(\frac{\lambda_{n+1}-\lambda_n-\varepsilon}{\lambda_{n+1}-\lambda_n}\right)^{-1}\frac{\varepsilon^2}{(\varepsilon+\gamma_{n+1})(\varepsilon+\gamma_n)}
\end{multline*}
so that
\begin{equation*}
a_n
	=\frac{\lambda_n-\lambda_0+\varepsilon}{\lambda_n-\lambda_0}
	\left(\prod_{\substack{p=1\\p\neq n}}^{\infty} \left(1+\frac{\varepsilon}{\lambda_p-\lambda_n}\right)\left(1-\frac{\varepsilon}{\lambda_p-\lambda_n}\right)\right)\frac{\varepsilon^2}{\gamma_{n+1}\gamma_n}.\qedhere
\end{equation*}
\end{proof}

\begin{lem}[Bounds for $a_n$]\label{lem:an_bound}
There exist $C>0$ such that  for every $\varepsilon>0$ and $\delta>0$, the following holds. For every $n\geq 1$,
\[
a_n\frac{\gamma_n\gamma_{n+1}}{\varepsilon^2}
+
a_n
+
a_n\frac{\gamma_{n+1}}{\varepsilon}
	+a_n\frac{\gamma_{n}}{\varepsilon}
	\leq C.
\]
\end{lem}

\begin{proof}
These inequalities are a direct consequence of the formula for $a_n$. Indeed, we have
\[
0\leq
	a_n\frac{\gamma_n\gamma_{n+1}}{\varepsilon^2}
	\leq 1+\frac{\varepsilon}{\lambda_n-\lambda_0}
	\leq 2.
\]
Similarly, since $\frac{\gamma_n}{\varepsilon+\gamma_n}=1+\frac{\varepsilon}{\lambda_{n-1}-\lambda_{n}}$ and  $\frac{\gamma_{n+1}}{\varepsilon+\gamma_{n+1}}=1-\frac{\varepsilon}{\lambda_{n+1}-\lambda_n}$, we get
\begin{equation*}
0
	\leq a_n
	\leq \left(1+\frac{\varepsilon}{\lambda_n-\lambda_0}\right)\left(1+\frac{\varepsilon}{\lambda_{n+1}-\lambda_n}\right)\left(1-\frac{\varepsilon}{\lambda_{n-1}-\lambda_{n}}\right)\frac{\varepsilon^2}{(\varepsilon+\gamma_{n+1})(\varepsilon+\gamma_n)}
	\leq 8,
\end{equation*}
\begin{equation*}
0
	\leq a_n\frac{\gamma_n}{\varepsilon}
	\leq \left(1+\frac{\varepsilon}{\lambda_n-\lambda_0}\right)\left(1-\frac{\varepsilon}{\lambda_{n-1}-\lambda_{n}}\right)\frac{\varepsilon}{\varepsilon+\gamma_n}
	\leq 4,
\end{equation*}
\begin{equation*}
0
	\leq a_n\frac{\gamma_{n+1}}{\varepsilon}
	\leq \left(1+\frac{\varepsilon}{\lambda_n-\lambda_0}\right)\left(1+\frac{\varepsilon}{\lambda_{n+1}-\lambda_n}\right)
	\frac{\varepsilon}{\varepsilon+\gamma_{n+1}}
	\leq 4.\qedhere
\end{equation*}
\end{proof}

Now, we remove the coefficients which are too far from the diagonal in the  sense that $|n-p|\geq \varepsilon^{-r}$ for some fixed parameter  $0<r<1$. We also justify that we can neglect the coefficients outside the region $\lambda_n+\varepsilon\in\Lambda_-(\delta)$.

\begin{rk}[Absolute convergence]\label{rk:toeplitz_miller}
In order to establish error bounds, we first prove absolute convergence of the summand. Thanks to Miller and Xu~\cite{MillerXu2011}, Lemma 4.7, we know the convergence of the series

\[
\sum_{\substack{m_1,\dots,m_{k}\in\Z\\ m_1+\dots+m_k=0}}\prod_{i=1}^{k}\frac{1}{|m_i|}
	=\frac{1}{2\pi}\int_{0}^{2\pi}g(\theta)^k\d\theta,
\]
where $g(\theta):=-\log(2(1-\cos(\theta))$ for $0<\theta<2\pi$ satisfies $g\in L^2_{r,0}(\T)$  and $\widehat{g}(k)=1/|k|$ for $k\neq 0$.
\end{rk}

\begin{lem}[Bounds for $\lambda_n-\lambda_p$]\label{lem:lambdanp}
For every $n,p\geq 0$ such that $n\neq p$, there holds
\[\frac{\varepsilon^2}{(\lambda_p-\lambda_n)^2}\leq\frac{1}{|p-n|^2}.
\]
Moreover, when $p\neq n+1$ one has
\begin{equation*}
\frac{\varepsilon}{|\lambda_p-\lambda_n-\varepsilon|}
	\leq \frac{2}{|p-n|},
\end{equation*}
whereas when $p=n+1$, one has
\begin{equation*}
\frac{1}{\lambda_{n+1}-\lambda_n-\varepsilon}
	=\frac{1}{\gamma_{n+1}}.
\end{equation*}
\end{lem}

\begin{proof}
The first claim comes from the formula for $p\geq n$
\[
\sum_{k=n+1}^p\gamma_k+(p-n)\varepsilon=\lambda_p-\lambda_n\geq (p-n)\varepsilon.
\]
To establish the second claim, we remark that when $p>n$,
\[
\lambda_p-\lambda_n-\varepsilon=(p-n-1)\varepsilon+\sum_{k=n+1}^p\gamma_k\geq(p-n-1)\varepsilon\geq 0,
\]
and when $p<n$,
\[
\lambda_p-\lambda_n-\varepsilon=(p-n-1)\varepsilon-\sum_{k=p+1}^n\gamma_k\leq(p-n-1)\varepsilon\leq 0.
\]
As a consequence, when $p\neq n$, we have proven that
$
|\lambda_p-\lambda_n-\varepsilon|\geq |p-n-1|\varepsilon\geq\frac{|p-n|}{2}\varepsilon.\qedhere
$
\end{proof}

\begin{prop}[Restriction to small eigenvalues close to the diagonal]\label{prop:M_ACV} Let $0<r<1$. There exists $C(\delta)>0$ such that for every $\varepsilon>0$,
\[
\varepsilon\sum_{\substack{n_1,\dots,n_{k+1}\geq 1\\ n_1=n_{k+1}} }\prod_{i=1}^{k}\sqrt{a_{n_i}\gamma_{n_{i+1}}\gamma_{n_i+1}}\frac{1}{|\lambda_{n_{i+1}}-\lambda_{n_i}-\varepsilon|}
	\leq C(\delta).
\]
Moreover, there holds
\begin{equation*}
\left|\widehat{u_0^{\varepsilon}}(k)
	-  \varepsilon\sum_{\substack{n_1,\dots,n_{k+1}\geq 1,\; n_1=n_{k+1}, \\ |n_i-n_{i+1}|\leq \varepsilon^{-r},\; \lambda_{n_i}+\varepsilon\in\Lambda_-(\delta)}}\prod_{i=1}^{k}\sqrt{a_{n_i}\gamma_{n_{i+1}}\gamma_{n_i+1}}\frac{1}{\lambda_{n_{i+1}}-\lambda_{n_i}-\varepsilon}e^{i(\theta_{n_i+1}-\theta_{n_{i+1}})}\right|
		\leq C(\delta)\varepsilon^{r}+C\delta.
\end{equation*}
\end{prop}

\begin{proof}
Let us in the proof denote the sum of absolute values
\[
S:=\varepsilon\sum_{\substack{n_1,\dots,n_{k+1}\geq 1\\ n_1=n_{k+1}} }\prod_{i=1}^{k}\sqrt{a_{n_{i}}\gamma_{n_{i+1}}\gamma_{n_i+1}}\frac{1}{|\lambda_{n_{i+1}}-\lambda_{n_i}-\varepsilon|}.
\]
We deduce from Lemma~\ref{lem:lambdanp} that
\begin{equation}\label{ineq:maj_S}
S
	\leq C\varepsilon\sum_{\substack{n_1,\dots,n_{k+1}\geq 1\\ n_1=n_{k+1}} }
	\left(\prod_{n_i\geq 1}\ \sqrt{a_{n_i}}\right)
	\left(\prod_{\substack{n_i\geq 1\\n_{i+1}\neq n_i+1}}\frac{\sqrt{\gamma_{n_{i+1}}\gamma_{n_{i}+1}}}{\varepsilon}\right)
	\left(\prod_{i=1}^k\frac{2}{|n_{i+1}-n_i|}\right).
\end{equation}
Using Lemma~\ref{lem:an_bound} on the bounds of $a_n$, one can note that every term of the form $\sqrt{a_{n_i}}\frac{\sqrt{\gamma_{n_{i}}\gamma_{n_{i}+1}}}{\varepsilon}$, $\sqrt{a_{n_i}}\frac{\sqrt{\gamma_{n_{i}+1}}}{\sqrt{\varepsilon}}$, $\sqrt{a_{n_i}}\frac{\sqrt{\gamma_{n_{i}}}}{\sqrt{\varepsilon}}$ and $\sqrt{a_{n_i}}$ is bounded by $C$.

We split the upper bound in several regions $|p-n|\geq\varepsilon^{-r}$, $\lambda_n+\varepsilon\not\in\Lambda(\delta)$, and $\lambda_n+\varepsilon\in\Lambda_+(\delta)$, for which we expect the sum to be small, and one region $\lambda_n+\varepsilon\in\Lambda_-(\delta)$, for which we expect the sum to be bounded.

\paragraph{Other eigenvalues}
We first assume that $\lambda_{n_1}+\varepsilon\not\in\Lambda(\delta)$, and denote
\[
S_{\Lambda^c}:=\varepsilon\sum_{\substack{n_1,\dots,n_{k+1}\geq 1\\ n_1=n_{k+1} \\ \lambda_{n_1}+\varepsilon\not\in\Lambda(\delta) } }	
	\prod_{i=1}^{k}\sqrt{a_{n_{i}}\gamma_{n_{i+1}}\gamma_{n_i+1}}
	\frac{1}{|\lambda_{n_{i+1}}-\lambda_{n_i}-\varepsilon|}.
\]
As for inequality~\eqref{ineq:maj_S}, Lemma~\ref{lem:lambdanp} gives
\[
S_{\Lambda^c}
	\leq C\varepsilon\sum_{\substack{n_1,\dots,n_{k+1}\geq 1\\ n_1=n_{k+1} \\ \lambda_{n_1}+\varepsilon\not\in\Lambda(\delta)} }\prod_{n_i\geq 1}\ \sqrt{a_{n_i}}\prod_{\substack{n_i\geq 1\\n_{i+1}\neq n_i+1}}\frac{\sqrt{\gamma_{n_{i+1}}\gamma_{n_{i}+1}}}{\varepsilon}\prod_{i=1}^k\frac{2}{|n_{i+1}-n_i|}.
\]
We make the change of variable $n=n_1$, $m_i=n_{i+1}-n_i$ for $1\leq i\leq k$. Then thanks to Lemma~\ref{lem:an_bound} bounding the terms involving $a_n$, there holds 
\[
S_{\Lambda^c}
	\leq  C \varepsilon \sum_{\substack{n\geq 1\\  \lambda_{n}+\varepsilon\not\in\Lambda(\delta)} }\sum_{\substack{m_1,\dots,m_k\in\Z \\ m_1+\dots+m_k=0}}\prod_{i=1}^k\frac{1}{|m_i|}.
\]
Using the bound from Remark~\ref{rk:toeplitz_miller}, we deduce that
\[
S_{\Lambda^c}
	\leq  C\varepsilon \sum_{\substack{n\geq 1\\  \lambda_{n}+\varepsilon\not\in\Lambda(\delta)} }1.
\] 
Finally, we note that there are at most $C\delta/\varepsilon$ possible indexes $n$ such that $\lambda_{n}+\varepsilon\not\in\Lambda(\delta)$, so that
\[
S_{\Lambda^c}
	\leq  C\delta.
\] 
The same applies if $n_1$ is replaced by some other index $n_i$ in the sum. In the following cases, we can therefore assume that $\lambda_{n_i}+\varepsilon\in \Lambda(\delta)$ for every $i$.

\paragraph{Off-diagonal terms} We now consider the terms such that $|n_2-n_1|\geq \varepsilon^{-r}$. Then assuming that $\varepsilon$ is small, we get from Lemma~\ref{lem:lambdanp} that
\[
|\lambda_{n_2}-\lambda_{n_1}-\varepsilon|\geq \frac 12\varepsilon^{1-r}.
\]
Let us denote
\[
S_{off}:=\varepsilon\sum_{\substack{n_1,\dots,n_{k+1}\geq 1\\ n_1=n_{k+1}\\ |n_{2}-n_1|> \varepsilon^{-r} \\ \forall j, \lambda_{n_j}+\varepsilon\in\Lambda(\delta)}}\prod_{i=1}^{k}\sqrt{a_{n_i}\gamma_{n_{i+1}}\gamma_{n_i+1}}
	\frac{1}{\left|\lambda_{n_{i+1}}-\lambda_{n_i}-\varepsilon\right|}.
\]
Using Lemma~\ref{lem:lambdanp} as in inequality~\eqref{ineq:maj_S} and the bound on $a_n$ from Lemma~\ref{lem:an_bound} we get the upper bound
\begin{equation*}
S_{off}
	\leq C\varepsilon^r \sum_{\substack{n_1,\dots,n_{k+1}\geq 1\\ n_1=n_{k+1} \\ \forall j,\lambda_{n_j}+\varepsilon\in\Lambda(\delta)} }\prod_{i=1}^{k}\sqrt{\gamma_{n_{i+1}}\gamma_{n_i+1}}\un_{n_{i+1}\neq n_i+1}.
\end{equation*}
We first remove the sum over $n_{i+1}$ when the condition $n_{i+1}=n_i+1$ is met. More precisely, we define $m_1$ as the smallest index $n_i$ such that $n_{i-1}+1\neq n_i$ (this is always possible because $n_1=n_{k+1}$), with the convention $n_0:=n_k$. Then let $0\leq d_1\leq k$ be such that for $0\leq i\leq d_1$, we have $n_{i}=n_1+i$ and $n_{1+d_1}\neq n_1+d_1$, we define $m_2:=n_{1+d_1}$ and so on by induction. The upper bound becomes
\begin{equation*}
S_{off}
	\leq C\varepsilon^r\sum_{l=1}^k\sum_{0\leq d_1,\dots,d_{l+1}\leq k} \sum_{\substack{m_1,\dots,m_{l+1}\geq 1\\ m_1=m_{l+1},\; m_{j+1}\neq m_j+1 \\ \forall j,\lambda_{m_j}+\varepsilon\in\Lambda(\delta)} }\prod_{i=1}^{l}\sqrt{\gamma_{m_{i+1}}\gamma_{m_i+d_i}}.
\end{equation*}
Since $\lambda_0=-\sum_{k\geq 1}\gamma_k\geq -\max(u_0)$ thanks to~\eqref{eq:lambda0}, we have
\[
\sum_{m_i}\sqrt{\gamma_{m_i}\gamma_{m_i+d_i}}
	\leq C.
\]
As a consequence,
\begin{equation*}
S_{off}
	\leq C\varepsilon^r.
\end{equation*}

The same applies if the condition $|n_2-n_1|\geq\varepsilon^{-r}$ is replaced by the condition $|n_{i+1}-n_i|\geq\varepsilon^{-r}$ for some $i$. In the following cases, we therefore assume that for every $i$, there holds $|n_{i+1}-n_i|\leq\varepsilon^{-r}$.

In this case, if $\varepsilon<\varepsilon_0(\delta)$, then when $\lambda_{n_1}+\varepsilon\in\Lambda_+(\delta)$ (resp. $\Lambda_-(\delta)$), there holds $\lambda_{n_i}+\varepsilon\in \Lambda_+(\delta/2)$ (resp. $\Lambda_-(\delta/2)$) for every $i$. Indeed, Definition~\ref{def:approximate_Birkhoff} with the parameter $\delta/2$ implies that there are at least $\frac{1}{C(\delta)\varepsilon}\geq\varepsilon^{-r}$ eigenvalues in $[-\max(u_0)+\delta/2,-\max(u_0)+\delta]$, and the same holds in $[-\min(u_0)+\delta/2,-\min(u_0)+\delta]$.

As a consequence, in the remaining cases, we can assume all the eigenvalues to be large (in $\Lambda_+(\delta)$) at the same time, or all the eigenvalues to be small (in $\Lambda_-(\delta)$)  at the same time. 

\paragraph{Large eigenvalues} Let
\[
S_+
	:=\varepsilon\sum_{\substack{n_1,\dots,n_{k+1}\geq 1,\; n_1=n_{k+1}\\ |n_{i+1}-n_i|\leq\varepsilon^{-r},\; \lambda_{n_i}+\varepsilon\in \Lambda_+(\delta)} } 
	\prod_{i=1}^{k}\sqrt{a_{n_i}\gamma_{n_{i+1}}\gamma_{n_i+1}}
	\frac{1}{\left|\lambda_{n_{i+1}}-\lambda_{n_i}-\varepsilon\right|}.
\]
Using Lemma~\ref{lem:lambdanp} as in inequality~\eqref{ineq:maj_S}, we know that
\begin{equation*}
S_+
	\leq C\varepsilon\sum_{\substack{n_1,\dots,n_{k+1}\geq 1,\; n_1=n_{k+1}\\ |n_{i+1}-n_i|\leq\varepsilon^{-r},\; \lambda_{n_i}+\varepsilon\in \Lambda_+(\delta)} }\prod_{n_i\geq 1}\ \sqrt{a_{n_i}}\prod_{\substack{n_i\geq 1\\n_{i+1}\neq n_i+1}}\frac{\sqrt{\gamma_{n_{i+1}}\gamma_{n_{i}+1}}}{\varepsilon}\prod_{i=1}^k\frac{2}{|n_{i+1}-n_i|}.
\end{equation*}

Since $n_{k+1}=n_1$, there exists an index $n_i$ such that $n_{i+1}\neq n_i+1$.
Up to multiplying the upper bound by some reordering constant $C$, one can assume that this index is $n_{1}$ so that the term $\sqrt{\gamma_{n_2}\gamma_{n_1}}$ appears in the upper bound. If $n_{i+1}=n_i+1$ for every $2\leq i\leq k$, then $n_1=n_{k+1}=n_2+k-1$. Otherwise, let $i_0$ be the first index $2\leq i\leq k$ such that $n_{i+1}\neq n_i+1$, we know that the term $\sqrt{\gamma_{n_{i_0+1}}\gamma_{n_{i_0}+1}}$ also appears in the upper bound, where $n_{i_0+1}=n_2+i_0-1$. As a consequence, we have proven that there exists $1\leq j_0=j_0(n_1,\dots,n_k)\leq k$ such that both $\sqrt{\gamma_{n_2}}$ and $\sqrt{\gamma_{n_2+j_0}}$ appear in the upper bound.

Then the bounds on $a_n$ from Lemma~\ref{lem:an_bound} imply
\begin{equation*}
S_+
	\leq C\varepsilon\sum_{\substack{n_1,\dots,n_{k+1}\geq 1,\; n_1=n_{k+1}\\ |n_{i+1}-n_i|\leq\varepsilon^{-r},\; \lambda_{n_i}+\varepsilon\in \Lambda_+(\delta)\\ n_2\neq n_1+1} } \frac{\sqrt{\gamma_{n_{2}}\gamma_{n_2+j_0}}}{\varepsilon}\prod_{i=1}^k\frac{2}{|n_{i+1}-n_i|}.
\end{equation*}
The bound from Remark~\ref{rk:toeplitz_miller}, coupled with the change of variable $n=n_2$, $m_i=n_{i+1}-n_i$ for $2\leq i\leq k+1$ (with the convention $n_{k+2}:=n_2$), leads to
\begin{equation*}
S_+
	\leq C\sum_{j_0=1}^k\sum_{\substack{n\geq 1\\ \lambda_{n}+\varepsilon\in \Lambda_+(\delta)}}\sqrt{\gamma_{n}\gamma_{n+j_0}}.
\end{equation*}
Finally, using Corollary~\ref{cor:asymptotics}, there holds 
\[
\sum_{\substack{n\geq 1\\ \lambda_n\in\Lambda_+(\delta)}}\gamma_n\leq C(\delta)\varepsilon\sqrt{\varepsilon}+C\delta.
\]
and we deduce
\[
S_+\leq C(\delta)\varepsilon\sqrt{\varepsilon}+C\delta.
\]

\paragraph{Small eigenvalues} In the last scenario, let 
\[
S_-
	:=\varepsilon\sum_{\substack{n_1,\dots,n_{k+1}\geq 1,\; n_1=n_{k+1}\\   |n_{i+1}-n_i|\leq\varepsilon^{-r},\; \lambda_{n_i}+\varepsilon\in\Lambda_-(\delta)}}
	\prod_{i=1}^{k}\sqrt{a_{n_i}\gamma_{n_{i+1}}\gamma_{n_i+1}}\frac{1}{\left|\lambda_{n_{i+1}}-\lambda_{n_i}-\varepsilon\right|}.
\]
We apply the argument from the former paragraph to get
\begin{equation*}
S_-
	\leq C\sum_{j_0=1}^k\sum_{\substack{n\geq 1\\ \lambda_{n}+\varepsilon\in \Lambda_-(\delta)}}\sqrt{\gamma_n\gamma_{n+j_0}}.
\end{equation*}
This is bounded by $C$ thanks to the lower bound~\eqref{eq:lambda0} on $\lambda_0$.

Summing the upper bounds for every one of the cases, we get the Proposition.
\end{proof}

\subsection{Approximation of the Fourier coefficients}

In this part, we express all the terms from the approximation of $\widehat{u_0^{\varepsilon}}(k)$ in Proposition~\ref{prop:M_ACV} as a function of $F(-\lambda_n)$ uniquely.

\begin{thm}[Fourier coefficients as a Riemann sum]\label{prop:Riemann_CV} 

Let $0<c<r<1$. For every $\delta>0$, there exist $C(\delta)>0, \varepsilon_0(\delta)>0$, and a function $R$ of $\varepsilon$ uniquely, tending to zero as $\varepsilon\to 0$, such that for every $0<\varepsilon<\varepsilon_0(\delta)$,
\begin{equation}\label{ineq:riemann}
\left|\widehat{u_0^{\varepsilon}}(k)
	-  \varepsilon\sum_{\substack{n\geq 1\\ \lambda_{n}+\varepsilon\in\Lambda_-(\delta)}} \sinc(k\pi F(-\lambda_n))e^{-ik\frac{x_+(-\lambda_n)+x_-(-\lambda_n)}{2}}\right|
	\leq C(\delta)(\varepsilon^{r-c}+\varepsilon^{1-2r})+R(\varepsilon)+C\delta.
\end{equation}
\end{thm}

The proof of Theorem~\ref{prop:Riemann_CV} decomposes in several steps. We first approximate $\lambda_p-\lambda_n$ and $a_n$ by functions of $F(-\lambda_n)$ only. Then we simplify the sum obtained by replacing the terms by their approximation.

\begin{lem}[Eigenvalues and distribution function]\label{lem:asymptotics}
Let $\lambda_n+\varepsilon,\lambda_p+\varepsilon\in\Lambda_-(\delta)$ such that $|p-n|\leq\varepsilon^{-r}$, then we have
\begin{equation*}
\left|\frac{F(-\lambda_n)^2}{(p-n)^2}-\frac{\varepsilon^2}{(\lambda_p-\lambda_n)^2}\right|
	\leq \frac{C(\delta)\sqrt{\varepsilon}}{|p-n|^{1+1/2r}}
\end{equation*}
and
\begin{equation*}
\left|\frac{F(-\lambda_n)}{p-n-F(-\lambda_n)}-\frac{\varepsilon}{\lambda_p-\lambda_n-\varepsilon}\right|
	\leq  C(\delta)\frac{\sqrt{\varepsilon}+\varepsilon^{1-2r}}{(p-n)^2}.
\end{equation*}
\end{lem}

\begin{proof}
We use Corollary~\ref{cor:asymptotics} in the small eigenvalue case and deduce that there exists $\xi_{n,p}\in[-\lambda_p,-\lambda_n]$ such that
\begin{equation}\label{eq:xi_np}
\left|(\lambda_p-\lambda_n)F(\xi_{n,p})-(p-n)\varepsilon\right|
	\leq C(\delta)\varepsilon\sqrt{\varepsilon}.
\end{equation}
As a consequence, we have from Lemma~\ref{lem:lambdanp} that
\begin{align*}
\left|\frac{F(\xi_{n,p})}{p-n}-\frac{\varepsilon}{\lambda_p-\lambda_n}\right|
	\leq \frac{C(\delta)\varepsilon\sqrt{\varepsilon}}{(p-n)(\lambda_p-\lambda_n)}
	\leq \frac{C(\delta)\sqrt{\varepsilon}}{(p-n)^2}.
\end{align*}
Using Lemma~\ref{lem:lambdanp} again, we know that $|\frac{\varepsilon}{\lambda_p-\lambda_n}|\leq \frac{1}{|p-n|}$, whereas $|\frac{F(\xi_{n,p})}{p-n}|\leq \frac{1}{|p-n|}$, and therefore
\[
\left|\frac{F(\xi_{n,p})^2}{(p-n)^2}-\frac{\varepsilon^2}{(\lambda_p-\lambda_n)^2}\right|
	\leq \frac{C(\delta)\sqrt{\varepsilon}}{|p-n|^3}.
\]
Then, we know from the Lipschitz properties of $F$ (see Corollary~\ref{cor:asymptotics}) that
\begin{align*}
|F(\xi_{n,p})-F(-\lambda_n)|
	\leq C(\delta)|\xi_{n,p}+\lambda_n|
	\leq C(\delta)|\lambda_p-\lambda_n|.
\end{align*}
However, thanks to~\eqref{eq:xi_np}, we have
\begin{equation}\label{eq:lambdanp}
|\lambda_p-\lambda_n|
	\leq \frac{C(\delta)\varepsilon\sqrt{\varepsilon}+|p-n|\varepsilon}{F(\xi_{n,p})}
	\leq C'(\delta)|p-n|\varepsilon,
\end{equation}
so that
\begin{equation}\label{eq:Fxi}
|F(\xi_{n,p})-F(-\lambda_n)|
	\leq C(\delta)|p-n|\varepsilon.
\end{equation}
As a consequence, one can replace $F(\xi_{n,p})$ by $F(-\lambda_n)$ up to a small error:
\begin{equation*}
\left|\frac{F(-\lambda_n)^2}{(p-n)^2}-\frac{\varepsilon^2}{(\lambda_p-\lambda_n)^2}\right|
	\leq \frac{C(\delta)\sqrt{\varepsilon}}{|p-n|^3}+\frac{C(\delta)\varepsilon}{|p-n|}.
\end{equation*}
Since $|p-n|\leq\varepsilon^{-r}$ with $r<1$, we get that $\varepsilon\leq \frac{1}{|p-n|^{1/r}}$ and we conclude that
\begin{equation*}
\left|\frac{F(-\lambda_n)^2}{(p-n)^2}-\frac{\varepsilon^2}{(\lambda_p-\lambda_n)^2}\right|
	\leq \frac{C(\delta)\sqrt{\varepsilon}}{|p-n|^{3}}+\frac{C(\delta)\sqrt{\varepsilon}}{|p-n|^{1+1/2r}}.
\end{equation*}

Similarly, using inequalities~\eqref{eq:xi_np} and~\eqref{eq:Fxi}, we have
\begin{align*}
\left|(\lambda_p-\lambda_n-\varepsilon)F(-\lambda_n)-(p-n-F(-\lambda_n))\varepsilon\right|
	&=\left|(\lambda_p-\lambda_n)F(-\lambda_n)-(p-n)\varepsilon\right|\\
	&\leq C(\delta)\varepsilon\sqrt{\varepsilon}+(\lambda_p-\lambda_n)C(\delta)|p-n|\varepsilon.
\end{align*}
Since inequality~\eqref{eq:lambdanp} implies $|\lambda_p-\lambda_n|\leq C(\delta)|p-n|\varepsilon\leq C(\delta)\varepsilon^{1-r}$, we finally get
\[
\left|(\lambda_p-\lambda_n-\varepsilon)F(-\lambda_n)-(p-n-F(-\lambda_n))\varepsilon\right|
	\leq C(\delta)(\varepsilon\sqrt{\varepsilon}+\varepsilon^{2-2r}),
\]
so that
\begin{equation*}
\left|\frac{F(-\lambda_n)}{(p-n-F(-\lambda_n))}-\frac{\varepsilon}{\lambda_p-\lambda_n-\varepsilon}\right|
	\leq \frac{ C(\delta)(\varepsilon\sqrt{\varepsilon}+\varepsilon^{2-2r})}{(p-n-F(-\lambda_n))(\lambda_p-\lambda_n-\varepsilon)}.
\end{equation*}
Since $\lambda_n+\varepsilon\in\Lambda_-(\delta)$, we know that $F(-\lambda_n)\leq 1-\frac{1}{C(\delta)}$. Moreover, we make use of Lemma~\ref{lem:lambdanp} and see that actually
\begin{equation*}
\left|\frac{F(-\lambda_n)}{(p-n-F(-\lambda_n))}-\frac{\varepsilon}{\lambda_p-\lambda_n-\varepsilon}\right|
	\leq  C(\delta)\frac{(\sqrt{\varepsilon}+\varepsilon^{1-2r})}{(p-n)^2}.\qedhere
\end{equation*}
\end{proof}

Then, we establish an approximation of $\frac{1}{\varepsilon}\sqrt{a_n\gamma_n\gamma_{n+1}}$.

\begin{lem}[Approximation of $a_n$ in terms of the distribution function]\label{lem:asymptotics_an} Let $0<c<r<1$. Then there exist $C(\delta)>0,\varepsilon_0(\delta)>0$ such that  for every $\delta>0$ and $0<\varepsilon<\varepsilon_0(\delta)$, the following holds. For every $n\geq 1$ satisfying $\lambda_n+\varepsilon\in\Lambda_-(\delta)$,
\[
\left|a_n\frac{\gamma_n\gamma_{n+1}}{\varepsilon^2}-
	\sinc(\pi F(-\lambda_n))^2\right|
	\leq C(\delta)(\varepsilon^{r-c}+\varepsilon^{1-r}).\qedhere
\]
\end{lem}

\begin{proof}
We consider the logarithm of $a_n$
\[
\log\left(a_n\frac{\gamma_n\gamma_{n+1}}{\varepsilon^2}\right)
	=\log\left(1+\frac{\varepsilon}{\lambda_n-\lambda_0}\right)
	\sum_{\substack{p=1\\p\neq n}}^{\infty} \log\left(1-\frac{\varepsilon^2}{(\lambda_p-\lambda_n)^2}\right).\]
When $|p-n|\geq \varepsilon^{-r}$, we make use of Lemma~\ref{lem:lambdanp}:
\[
\frac{\varepsilon^2}{(\lambda_p-\lambda_n)^2}
	\leq \frac{1}{|n-p|^2}
	\leq \frac{\varepsilon^{r-c}}{|n-p|^{1+c/r}}.
\]
When $|p-n|\leq \varepsilon^{-r}$, since $\lambda_n+\varepsilon\in\Lambda_-(\delta)$, then $\lambda_p+\varepsilon\in\Lambda_-(\delta/2)$ and using Lemma~\ref{lem:asymptotics} there holds
\begin{equation*}
\left|\frac{F(-\lambda_n)^2}{(p-n)^2}-\frac{\varepsilon^2}{(\lambda_p-\lambda_n)^2}\right|
	\leq \frac{C(\delta)\varepsilon^{1/2}}{(p-n)^{1+1/2r}}.
\end{equation*}
Hence, we get by summation that
\[
\left|\log\left(a_n\frac{\gamma_n\gamma_{n+1}}{\varepsilon^2}\right)
	-\sum_{\substack{p=1\\|p-n|\leq\varepsilon^{-r}}}^{\infty} \log\left(1-\frac{F(-\lambda_n)^2}{(p-n)^2}\right)\right|
	\leq C(\delta)(\varepsilon^{1/2}+\varepsilon^{r-c}).
\]
Note that $F(-\lambda_n)\leq 1-1/C(\delta)$, as a consequence,
\[
\left|\log\left(a_n\frac{\gamma_n\gamma_{n+1}}{\varepsilon^2}\right)\right|
	\leq C(\delta).
\]
Also note that since $\lambda_n+\varepsilon\in\Lambda_-(\delta)$, we have $n\geq\frac{\delta}{C\varepsilon}$, therefore the indexes $p\geq 1$ such that $|n-p|\leq\varepsilon^{-r}$ is the same as the indexes $p\in\Z$ such that $|n-p|\leq\varepsilon^{-r}$ when $\varepsilon<\varepsilon_0(\delta)$. 
Moreover, since $F(-\lambda_n)\leq 1-1/C(\delta)$, the change of variable $k=p-n$ for $p> n$ and $k=n-p$ for $n>p$ leads to
\begin{equation*}
\sum_{\substack{p=-\infty\\|p-n|>\varepsilon^{-r}}}^{\infty} \left|\log\left(1-\frac{F(-\lambda_n)^2}{(p-n)^2}\right)\right|
	\leq 2\sum_{\substack{k=\varepsilon^{-r}+1}}^{\infty} \left|\log\left(1-\frac{F(-\lambda_n)^2}{k^2}\right)\right|
	\leq C(\delta)\varepsilon^{1-r}.
\end{equation*}
Therefore, we have proven that
\[
\left|\log\left(a_n\frac{\gamma_n\gamma_{n+1}}{\varepsilon^2}\right)
	-2\sum_{k=1}^{\infty} \log\left(1-\frac{F(-\lambda_n)^2}{(p-n)^2}\right)\right|
	\leq C(\delta)(\varepsilon^{1/2}+\varepsilon^{r-c}+\varepsilon^{1-r}).
\]
Taking the exponential, since $\left|\log\left(a_n\frac{\gamma_n\gamma_{n+1}}{\varepsilon^2}\right)\right|$ stays bounded by $C(\delta)$ for $\varepsilon<\varepsilon_0(\delta)$, we deduce
\[
\left|a_n\frac{\gamma_n\gamma_{n+1}}{\varepsilon^2}-
	\prod_{\substack{k=1}}^{\infty} \left(1-\frac{F(-\lambda_n)^2}{k^2}\right)^2\right|
	\leq C(\delta)(\varepsilon^{1/2}+\varepsilon^{r-c}+\varepsilon^{1-r}).
\]

Finally, we use the Weierstrass sine product formula for $z\in(0,1)$ 
\[
\sinc(\pi z)=\frac{\sin(\pi z)}{\pi z}= \prod_{k\geq 1}\left(1-\frac{z^2}{k^2}\right)
\]
to deduce that
\[
\left|a_n\frac{\gamma_n\gamma_{n+1}}{\varepsilon^2}-
	\sinc(\pi F(-\lambda_n))^2\right|
	\leq C(\delta)(\varepsilon^{1/2}+\varepsilon^{r-c}+\varepsilon^{1-r}),
\]
where $\min(r-c,1-2r)<1/2$.
\end{proof}

In what follows, the above approximations will lead us to study the series
\begin{equation*}
\sum_{\substack{m_1,\dots,m_{k}\in\Z\\m_1+\dots+m_{k}=0}}\prod_{i=1}^{k}\frac{1}{m_i-F(-\lambda_{n})}.
\end{equation*}
We first prove the convergence  and find a formula for this sum.
\begin{lem}[Toeplitz identity]\label{lem:Toeplitz}
Let $c\in (0,1)$. Then 
\[
\sum_{\substack{m_1,\dots,m_{k}\in\Z\\ m_1+\dots+m_k=0}}\prod_{i=1}^{k}\frac{1}{|m_i-c|}
	<\infty
\] 
and there holds
\begin{equation*}
\sum_{\substack{m_1,\dots,m_{k}\in\Z\\ m_1+\dots+m_k=0}}\prod_{i=1}^{k}\frac{1}{m_i-c}
	=(-1)^k\frac{\pi^{k-1}\sin(k\pi c)}{k c\sin(\pi c)^k}.
\end{equation*}
\end{lem}

\begin{proof}
To establish absolute convergence, we use the inequality $|m_i-c|\geq |m_i|\min(c,1-c)$ to get that 
\[
\sum_{\substack{m_1,\dots,m_{k}\in\Z\\ m_1+\dots+m_k=0\\\exists i,|m_i|>\varepsilon^{-r}}}\prod_{i=1}^{k}\frac{1}{|m_i-c|}
	\leq \frac{1}{\min(c,1-c)^k}\sum_{\substack{m_1,\dots,m_{k}\in\Z\\ m_1+\dots+m_k=0\\\exists i,|m_i|>\varepsilon^{-r}}}\prod_{i=1}^{k}\frac{1}{|m_i|},
\]
where the upper bound is the remainder term of an absolutely convergent series thanks to Lemma 4.7 in Miller and Xu~\cite{MillerXu2011}, see Remark~\ref{rk:toeplitz_miller}.

We now define
\[
f(x):=\sum_{m\in\Z}\frac{e^{imx}}{m-c}.
\]
This Fourier series is convergent in $L^2(\T)$, and one can check by direct calculation of the Fourier coefficients $\widehat{f}(n)=\frac{1}{2\pi}\int_0^{2\pi}f(x)e^{-inx}\d x$ that it is equal to
\[
f(x)=-\frac{\pi e^{ic(x-\pi)}}{\sin(\pi c)}.
\]
Indeed, for $n\in\Z$, the above expression leads to
\[
\widehat{f}(n)
	=-\frac{\pi e^{-ic\pi}}{\sin(\pi c)}\frac{1}{2\pi}\int_0^{2\pi}e^{i(c-n)x}\d x
	=-\frac{ e^{-ic\pi}}{2\sin(\pi c)}\left[\frac{e^{i(c-n)x}}{i(c-n)}\right]_0^{2\pi}
	=-\frac{1}{c-n}.
\]
Since $f$ belongs to $L^2$ and $f^k$ also belongs to $L^2$ for every positive integer $k$, the convolution theorem implies
\[
\sum_{\substack{m_1,\dots,m_k \in \Z\\ m_1+\dots+m_k=0}}\widehat{f}(m_1)\dots \widehat{f}(m_k)=\widehat{f^k}(0)=\frac{1}{2\pi}\int_0^{2\pi}f(x)^k\d x.
\]
As a consequence, we conclude that
\begin{equation*}
\sum_{\substack{m_1,\dots,m_k \in \Z\\ m_1+\dots+m_k=0}}\frac{1}{(m_1-c)\dots(m_k-c)}
	=(-1)^k\frac{\pi^{k-1}e^{-ik\pi c}}{2\sin(\pi c)^k}\left[\frac{e^{ikcx}}{ikc}\right]_0^{2\pi}\\
	=(-1)^k\frac{\pi^{k-1}\sin(k\pi c)}{kc\sin(\pi c)^k}.\qedhere
\end{equation*}
\end{proof}

\begin{proof}[Proof of Theorem~\ref{prop:Riemann_CV}]
We first remove the off-diagonal coefficients and the eigenvalues which are not in $\Lambda_-(\delta)$ thanks to Proposition~\ref{prop:M_ACV} up to adding a remainder term of the form $C(\delta)\varepsilon^r+C\delta$: we get
\begin{equation*}
\left|\widehat{u_0^{\varepsilon}}(k)
	-  \varepsilon\sum_{\substack{n_1,\dots,n_{k+1}\geq 1,\; n_1=n_{k+1}, \\ |n_i-n_{i+1}|\leq \varepsilon^{-r},\; \lambda_{n_i}+\varepsilon\in\Lambda_-(\delta)}}\prod_{i=1}^{k}\sqrt{a_{n_i}\gamma_{n_{i+1}}\gamma_{n_i+1}}\frac{1}{\lambda_{n_{i+1}}-\lambda_{n_i}-\varepsilon}e^{i(\theta_{n_i+1}-\theta_{n_{i+1}})}\right|
		\leq C(\delta)\varepsilon^{r}+C\delta.
\end{equation*}
A re-indexation implies
\begin{equation*}
\left|\widehat{u_0^{\varepsilon}}(k)
	-  \varepsilon\sum_{\substack{n_1,\dots,n_{k+1}\geq 1,\; n_1=n_{k+1}, \\ |n_i-n_{i+1}|\leq \varepsilon^{-r},\; \lambda_{n_i}+\varepsilon\in\Lambda_-(\delta)}}\prod_{i=1}^{k}\sqrt{a_{n_i}\gamma_{n_{i}}\gamma_{n_i+1}}\frac{1}{\lambda_{n_{i+1}}-\lambda_{n_i}-\varepsilon}e^{i(\theta_{n_i+1}-\theta_{n_{i}})}\right|
		\leq C(\delta)\varepsilon^{r}+C\delta.
\end{equation*}

We then use the second inequality from Lemma~\ref{lem:asymptotics}
\begin{equation*}
\left|\frac{F(-\lambda_n)}{p-n-F(-\lambda_n)}-\frac{\varepsilon}{\lambda_p-\lambda_n-\varepsilon}\right|
	\leq  C(\delta)\frac{\sqrt{\varepsilon}+\varepsilon^{1-2r}}{(p-n)^2}.
\end{equation*}
When $\lambda_n+\varepsilon,\lambda_p+\varepsilon\in\Lambda_-(\delta)$ and $|n-p|\leq\varepsilon^{-r}$, we know thanks to the Lipschitz properties of $F$ and inequality~\eqref{eq:lambdanp} that
\[
|F(-\lambda_n)-F(-\lambda_p)|
	\leq C(\delta)|\lambda_p-\lambda_n|
	\leq C'(\delta)|p-n|\varepsilon
	\leq C'(\delta)\varepsilon^{1-r}.
\]
As a consequence, we have that if $|n-n_1|,|p-n_1|\leq\varepsilon^{-r}$, then
\begin{equation*}
\left|\frac{F(-\lambda_{n_1})}{p-n-F(-\lambda_{n_1})}-\frac{\varepsilon}{\lambda_p-\lambda_n-\varepsilon}\right|
	\leq  C(\delta)\frac{\sqrt{\varepsilon}+\varepsilon^{1-2r}}{(p-n)^2}+C(\delta)\frac{\varepsilon^{1-r}}{|p-n|}
	\leq C'(\delta)\frac{\sqrt{\varepsilon}+\varepsilon^{1-2r}}{|p-n|}.
\end{equation*}
Using the bound on $a_n\gamma_n\gamma_{n+1}/\varepsilon$ from Lemma~\ref{lem:an_bound}, we get
\begin{multline*}
\left|\widehat{u_0^{\varepsilon}}(k)
	-  \varepsilon\sum_{\substack{n_1,\dots,n_{k+1}\geq 1,\; n_1=n_{k+1} \\ |n_i-n_{i+1}|\leq \varepsilon^{-r},\; \lambda_{n_i}+\varepsilon\in\Lambda_-(\delta)}}\prod_{i=1}^{k}\frac{\sqrt{a_{n_i}\gamma_{n_i}\gamma_{n_i+1}}}{\varepsilon}\frac{ F(-\lambda_{n_1})}{n_{i+1}-n_i-F(-\lambda_{n_1})}e^{i(\theta_{n_i+1}-\theta_{n_{i}})}\right|\\
		\leq C(\delta)\varepsilon^r+C\delta+C(\delta)\varepsilon(\sqrt{\varepsilon}+\varepsilon^{1-2r})\sum_{\substack{n_1,\dots,n_{k+1}\geq 1,\; n_1=n_{k+1} \\ |n_i-n_{i+1}|\leq \varepsilon^{-r},\; \lambda_{n_i}+\varepsilon\in\Lambda_-(\delta)}}\prod_{i=1}^{k}\frac{1}{|n_{i+1}-n_i|}.
\end{multline*}
Since there are not more than $\frac{C}{\varepsilon}$ indexes $n$ such that $\lambda_n+\varepsilon\in\Lambda_-(\delta)$, this leads to
\begin{multline*}
\left|\widehat{u_0^{\varepsilon}}(k)
	-  \varepsilon\sum_{\substack{n_1,\dots,n_{k+1}\geq 1,\; n_1=n_{k+1} \\ |n_i-n_{i+1}|\leq \varepsilon^{-r},\; \lambda_{n_i}+\varepsilon\in\Lambda_-(\delta)}}\prod_{i=1}^{k}\frac{\sqrt{a_{n_i}\gamma_{n_i}\gamma_{n_i+1}}}{\varepsilon}\frac{ F(-\lambda_{n_1})}{n_{i+1}-n_i-F(-\lambda_{n_1})}e^{i(\theta_{n_i+1}-\theta_{n_{i}})}\right|\\
		\leq C(\delta)(\varepsilon^r+\sqrt{\varepsilon}+\varepsilon^{1-2r})+C\delta.
\end{multline*}

Next, for every $n$ such that $\lambda_n+\varepsilon\in\Lambda_-(\delta)$, we use the approximation of $a_{n}$ from Lemma~\ref{lem:asymptotics_an} 
\[
\left|a_n\frac{\gamma_n\gamma_{n+1}}{\varepsilon^2}-
	\sinc(\pi F(-\lambda_n))^2\right|
	\leq C(\delta)(\varepsilon^{r}+\varepsilon^{1-2r}).
\]
We  also note that $x\mapsto\operatorname{sinc}(x)=\frac{\sin(x)}{x}$ is Lipschitz on $\R$ and $F$ is $C(\delta)$-Lipschitz on $[-\max(u_0)+\delta,-\min(u_0)-\delta]$. Therefore, for every $n,p$ satisfying $|n-p|\leq\varepsilon^{-r}$ and $\lambda_n+\varepsilon,\lambda_p+\varepsilon\in\Lambda_-(\delta)$, the  use of inequality~\eqref{eq:lambdanp} leads to
\[
|\operatorname{sinc}(\pi F(-\lambda_n))-\operatorname{sinc}(\pi F(-\lambda_p))|	
	\leq C(\delta)|\lambda_p-\lambda_n|
	\leq C'(\delta)|n-p|\varepsilon
	\leq C'(\delta)\varepsilon^{1-r}.
\]
We have proven that for $|n-p|\leq\varepsilon^{-r}$ satisfying $\lambda_n+\varepsilon,\lambda_p+\varepsilon\in \Lambda_-(\delta)$, we have
\[
\left|a_p\frac{\gamma_p\gamma_{p+1}}{\varepsilon^2}-
	\sinc(\pi F(-\lambda_n))^2\right|
	\leq C(\delta)(\varepsilon^{r}+\varepsilon^{1-2r}).
\]
Using the Lipschitz properties of $x_+$ and $x_-$ in Corollary~\ref{cor:asymptotics} and the formula~\eqref{eq:phase} for the phase constants, we even have
\[
\left|a_p\frac{\gamma_p\gamma_{p+1}}{\varepsilon^2}e^{i(\theta_{p+1}-\theta_p)}-
	\sinc(\pi F(-\lambda_n))^2e^{i(\theta_{n+1}-\theta_n)}\right|
	\leq C(\delta)(\varepsilon^{r}+\varepsilon^{1-2r}).
\]
Since $\sinc$ is nonnegative on $[0,\pi]$, we get by summation that
\begin{multline*}
\left|\widehat{u_0^{\varepsilon}}(k)
	-  \varepsilon\sum_{\substack{n_1,\dots,n_{k+1}\geq 1,\; n_1=n_{k+1} \\ |n_i-n_{i+1}|\leq \varepsilon^{-r},\;  \lambda_{n_i}+\varepsilon\in\Lambda_-(\delta)}}\sinc(\pi F(-\lambda_{n_1}))^ke^{ik(\theta_{n_1+1}-\theta_{n_1})}\prod_{i=1}^{k}  \frac{ F(-\lambda_{n_1})}{n_{i+1}-n_i-F(-\lambda_{n_1})}\right|\\
		\leq C\delta+C(\delta)\varepsilon(\varepsilon^{r-c}+\varepsilon^{1-2r}) \sum_{\substack{n_1,\dots,n_{k+1}\geq 1,\; n_1=n_{k+1} \\ |n_i-n_{i+1}|\leq \varepsilon^{-r},\;  \lambda_{n_i}+\varepsilon\in\Lambda_-(\delta)}}\prod_{i=1}^{k}  \frac{ F(-\lambda_{n_1})}{|n_{i+1}-n_i-F(-\lambda_{n_1})|} .
\end{multline*}
 We use the absolute convergence from the Toeplitz sum in Lemma~\ref{lem:Toeplitz} to treat the sum over indexes $n_2,\dots,n_{k+1}$. Moreover, we know that there are no more than $C/\varepsilon$ terms in the sum over $n_1$, so that
\begin{multline}\label{ineq:n_to_n1}
\left|\widehat{u_0^{\varepsilon}}(k)
	-  \varepsilon\sum_{\substack{n_1,\dots,n_{k+1}\geq 1,\; n_1=n_{k+1} \\ |n_i-n_{i+1}|\leq \varepsilon^{-r},\; \lambda_{n_i}+\varepsilon\in\Lambda_-(\delta) }}
	\sinc(\pi F(-\lambda_{n_1}))^ke^{ik(\theta_{n_1+1}-\theta_{n_1})}
	\prod_{i=1}^{k} \frac{ F(-\lambda_{n_1})}{n_{i+1}-n_i-F(-\lambda_{n_1})}\right|\\
		\leq C(\delta)(\varepsilon^{r-c}+\varepsilon^{1-2r})+C\delta.
\end{multline}
Finally, we make the change of variable $n=n_1$ and $m_i=n_{i+1}-n_i$ for $1\leq i\leq k$. It only remains to use the Toeplitz identity from Lemma~\ref{lem:Toeplitz} to conclude that
\begin{multline*}
\left|\widehat{u_0^{\varepsilon}}(k)
	-  \varepsilon\sum_{\substack{n\geq 1\\ \lambda_n+\varepsilon\in\Lambda_-(\delta)}}\left(\frac{\sin(\pi F(-\lambda_{n}))}{\pi }\right)^{k}e^{ik(\theta_{n+1}-\theta_{n})}\frac{(-1)^k\pi^{k-1}\sin(k\pi F(-\lambda_n))}{k F(-\lambda_n)\sin(\pi F(-\lambda_n))^k}\right|\\
	\leq C(\delta)(\varepsilon^{r-c}+\varepsilon^{1-2r})+R(\varepsilon)+C\delta,
\end{multline*}
where $R(\varepsilon)$ is bounded by the remainder term in the Toeplitz identity from Lemma~\ref{lem:Toeplitz}
\begin{align*}
R(\varepsilon)
	&=C\varepsilon\sum_{\substack{n\geq 1\\ \lambda_n+\varepsilon\in\Lambda_-(\delta)}}
	\sum_{\substack{m_1,\dots,m_{k+1}\geq 1\\ \exists i, |m_i|> \varepsilon^{-r}}}\prod_{i=1}^k\frac{1}{|m_i-F(\lambda_{n_1})|} \left(\frac{\sin(\pi F(-\lambda_{n}))}{\pi }\right)^{k}\frac{\pi^{k-1}|\sin(k\pi F(-\lambda_n))|}{k F(-\lambda_n)\sin(\pi F(-\lambda_n))^k} \\
	&\leq C'\sum_{\substack{m_1,\dots,m_{k+1}\geq 1\\ \exists i, |m_i|> \varepsilon^{-r}}}\prod_{i=1}^k\frac{1}{|m_i-F(\lambda_{n_1})|}.
\end{align*}
We conclude by using that $(-1)^ke^{ik(\theta_{n+1}-\theta_n)}=e^{-ik\frac{x_+(-\lambda_n)+x_-(-\lambda_n)}{2}}$.
\end{proof}

\subsection{Link with Burgers equation and time evolution}\label{subsection:time_evolution}

In this part, we deduce an approximation of the $k$-th Fourier coefficient of the solution $u^{\varepsilon}(t)$ at time $t$ which is coherent with Proposition~\ref{prop:formule_uk}.

\begin{thm}[Fourier coefficients and Burgers equation]\label{thm:link_burgers}
 Let $k\in\Z$. Let $u^{\varepsilon}$ be the solution to~\eqref{eq:bo} with initial data $u_0^{\varepsilon}$. For every $T>0$, there exists $\varepsilon_0(\delta,T)$ such that for every $\varepsilon<\varepsilon_0(\delta,T)$ and $t\in[0,T]$, there holds
\[
\left|\widehat{u^{\varepsilon}(t)}(k)
	+\frac{i}{2 k\pi}\int_{\min(u_0)}^{\max(u_0)}e^{-ik(x_+(\eta)+2\eta t)}-e^{-ik(x_-(\eta)+2\eta t)}\d\eta\right|
		\leq C\delta.
\]
\end{thm}

\begin{proof}Fix $0<c<r<1$. 
We first consider the function $u_0^{\varepsilon}$ without any time evolution. We justify that we can pass to the limit in the Riemann sum
\begin{equation*}
\left|\widehat{u_0^{\varepsilon}}(k)
	-  \varepsilon\sum_{\substack{n\geq 1\\ \lambda_{n}+\varepsilon\in\Lambda_-(\delta)}} e^{-ik\frac{x_+(-\lambda_n)+x_-(-\lambda_n)}{2}}\sinc(k\pi F(-\lambda_n))\right|
	\leq C(\delta)(\varepsilon^{r-c}+\varepsilon^{1-2r})+R(\varepsilon)+C\delta.
\end{equation*}
The function $\eta\mapsto\sinc(k\pi F(\eta))$ is $\classeC^1$ on $[-\beta+\delta,\beta-\delta]$. Moreover, in the region $\eta+\varepsilon\in\Lambda_-(\delta)$, there holds $F(\eta)\geq 1/C(\delta)$. Since $\lambda_n+\varepsilon\in\Lambda_-(\delta)$, Corollary~\ref{cor:asymptotics} implies that $|\lambda_{n+1}-\lambda_n|\leq C(\delta)\varepsilon$, so that the mesh is tending to zero as $\varepsilon\to0$. More precisely, there holds thanks to~\eqref{eq:xi_np}, \eqref{eq:lambdanp} and~\eqref{eq:Fxi} that
\[
|(\lambda_{n+1}-\lambda_n)F(-\lambda_n)-\varepsilon|
	\leq C(\delta)\varepsilon\sqrt{\varepsilon},
\]
 so that $F(-\lambda)$ is the distribution function of the $\eta=-\lambda$'s. There are at most $\frac{C}{\varepsilon}$ indexes $n$ such that $\lambda_n+\varepsilon\in\Lambda_-(\delta)$, therefore we get by summation
\begin{multline*}
\left|\widehat{u_0^{\varepsilon}}(k)
	-  \sum_{\substack{n\geq 1\\ \lambda_{n}+\varepsilon\in\Lambda_-(\delta)}} e^{-ik\frac{x_+(-\lambda_n)+x_-(-\lambda_n)}{2}}\sinc(k\pi F(-\lambda_n))F(-\lambda_n)(\lambda_{n+1}-\lambda_{n})\right|\\
	\leq C(\delta)(\varepsilon^{r-c}+\varepsilon^{1-2r})+R(\varepsilon)+C\delta.
\end{multline*}
Passing to the limit $\varepsilon\to0$, this leads to
\[
\left|\widehat{u_0^{\varepsilon}}(k)
	+\frac{1}{k\pi}\int_{\min(u_0)+\delta}^{ \max(u_0)-\delta}e^{-ik\frac{x_+(\eta)+x_-(\eta)}{2}}\sin(k\pi F(\eta ))\d\eta\right|
		\leq C\delta.
\]
Finally, we use the definition $2\pi F(\eta)=x_+(\eta)-x_-(\eta)$ and simplify
\[
\left|\widehat{u_0^{\varepsilon}}(k)
	+\frac{1}{2i k\pi}\int_{\min(u_0)+\delta}^{\max(u_0)-\delta}e^{-ikx_-(\eta)}-e^{-ikx_+(\eta)}\d\eta\right|
	\leq C\delta.
\]
Given that the integrand is bounded by $2$, one can remove the $\delta$ in the integration bounds up to increasing $C$, and hence we get the result.

\paragraph{Time evolution} Let us now add the time into account. Let $u^{\varepsilon}$ be the solution to~\eqref{eq:bo} with initial data $u^{\varepsilon}(t=0)=u_0^{\varepsilon}$. We check that $M_{n,p}(u^{\varepsilon}(t);\varepsilon)$ becomes 
\[
M_{n,p}(u^{\varepsilon}(t);\varepsilon)
	=M_{n,p}(u_0^{\varepsilon};\varepsilon)^{i(\omega_{n+1}(u_0^{\varepsilon};\varepsilon)-\omega_p(u_0^{\varepsilon}; \varepsilon))t}.
\]
To find the formula for $\omega_n(u_0^{\varepsilon};\varepsilon)$, one can observe that
\[
v(t,x):=\frac{1}{\varepsilon}u^{\varepsilon}\left(\frac{t}{\varepsilon},x\right)
\]
is the solution to~\eqref{eq:bo} with parameter $\varepsilon=1$ and initial data $u_0^{\varepsilon}/\varepsilon$. But $f_n(v(t);1)=f_n(u^{\varepsilon}(t/\varepsilon);\varepsilon)$ so that since $\gamma_n(u_0^{\varepsilon};\varepsilon)=\varepsilon\gamma_n(v_0^{\varepsilon};1)$ and
\(
\zeta_n(u^{\varepsilon}(t);\varepsilon)
	=\sqrt{\varepsilon}\zeta_n(v(\varepsilon t);1),
\)
then Proposition 8.1 from~\cite{GerardKappeler2019} implies
\[
\zeta_n(u^{\varepsilon}(t);\varepsilon)
	=\sqrt{\varepsilon}\zeta_n(v_0^{\varepsilon};1)e^{i\omega_n(v_0^{\varepsilon};1)\varepsilon t}.
\]
Therefore,
\begin{align*}
\omega_n(u_0^{\varepsilon};\varepsilon)
	&=\varepsilon\omega_n(v_0^{\varepsilon};1)\\
	&=\varepsilon\left( n^2-2\sum_{k=n+1}^{\infty}\min(k,n)\gamma_k(v_0^{\varepsilon};1)\right)\\
	&=\varepsilon n^2-2\sum_{k=n+1}^{\infty}\min(k,n)\gamma_k(u_0^{\varepsilon};\varepsilon).
\end{align*}

As a consequence, we get the approximate solution at time $t$ by replacing every phase constant $\theta_n$ by $\theta_n+\omega_nt$, with $\omega_n=\omega_n(u_0^{\varepsilon};\varepsilon)$. Let us establish the Lipschitz properties of these new phase constants.
We have
\begin{equation*}
\omega_{n+1}-\omega_n
	=\varepsilon (2n+1)-2\sum_{k=n+1}^{\infty}\gamma_k
	= 2\lambda_n+\varepsilon.
\end{equation*}
But when $\lambda_n+\varepsilon,\lambda_p+\varepsilon\in\Lambda_-(\delta)$ and $|p-n|\leq \varepsilon^{-r}$, we have
\[
(\omega_{p+1}-\omega_p)-(\omega_{n+1}-\omega_{n})=2(\lambda_p-\lambda_n).
\]
Using inequality~\eqref{eq:lambdanp}, one deduces that
\[
\left|(\omega_{p+1}-\omega_p)-2\lambda_n\right|\leq C(\delta)\varepsilon^{1-r}.
\]
As a consequence, one has
\[
|\exp(i(\omega_{p+1}-\omega_p)t)-\exp(2i\lambda_n t)|
	\leq C(\delta)\varepsilon^{1-r}t.
\]
We use our above approximation approach by replacing the phase factors $e^{i(\theta_{n_{i+1}}-\theta_{n_i})}$ by  their time evolution $e^{i(\theta_{n_{i+1}}-\theta_{n_i}+(\omega_{n_{i+1}}-\omega_{n_i})t)}$ in the series.
Passing to the limit in the Riemann sum, for every $T>0$, there exists $\varepsilon_0(\delta,T)$ such that for every $\varepsilon<\varepsilon_0(\delta,T)$ and $t\in[0,T]$, there holds
\[
\left|
\widehat{u^{\varepsilon}(t)}(k)
	-\frac{i}{2 k\pi}\int_{\min(u_0)+\delta}^{\max(u_0)-\delta}e^{-ik(x_-(\eta)+2\eta t)}-e^{-ik(x_+(\eta)+2\eta t)}\d\eta\right|\leq C\delta.
\]
Given that the integrand is bounded, one can remove the $\delta$ in the integration bounds up to increasing $C$, and hence we get the result.
\end{proof}

We now make the link between Theorem~\ref{thm:link_burgers} and the Fourier coefficients of the signed sum of branches $u_{alt}^B$ for the multivalued solution $u^B$ to the Burgers equation obtained with the method of characteristics, see Figure~\ref{fig:multivalued2}. Every point $u^B$ is an image of the solution at time $t$ with abscissa $x$ as soon as it solves the implicit equation
\[
u^B=u_0(x-2u^Bt).
\]
On the real line (see a more detailed description in~\cite{MillerXu2011}), new sheets are formed at the breaking points $(t_{\xi},x_{\xi})$ such that $\xi$ is an inflection point $u_0'(\xi)\neq 0$, $u_0''(\xi)=0$, and
\[
(t_{\xi},x_{\xi})=\left(\frac{-1}{2u_0'(\xi)},\xi-\frac{u_0(\xi)}{u_0'(\xi)}\right).
\]
In the case of a single well potential, there are two such inflection points $\xi_{\pm}$, for which $u_0''(\xi_{\pm})=0$ and $u_0'''(\xi_{\pm})\neq 0$.  We assume that $u_0'(\xi_+)<0<u_0'(\xi_-)$. Right after the positive breaking time $t_{\xi_+}=\frac{-1}{2u_0'(\xi_+)}$, two new branches emerge, so that there are three branches in total. Because of the periodicity, this can lead to more branches as $t$ increases, see Figure~\ref{fig:multivalued}, that we denote $u_0^B(t,x)<\dots<u_{2P(t,x)}^B(t,x)$. We have defined the signed sum of branches in~\eqref{eq:u_alt} as
\[
u^B_{alt}(t,x)=\sum_{n=0}^{2P(t,x)}(-1)^nu_n^B(t,x).
\]
These branches are described by two (actual) functions $v_0^B$ and $v_1^B$ defined on subintervals $[X_-(t)-2\pi,X_+(t)]$ and $[X_-(t),X_+(t)]$ of the real line, see Figure~\ref{fig:multivalued2}.

More precisely, we consider one period of the initial data $u_0$, for which we follow the method of characteristics on $\R$. Then the solution to the corresponding multivalued Burgers equation on $\R$ has between $0$ and $3$ branches
\[
u(y+2u_0(y)t)=u_0(y)
\]
\[
u(x_{\pm}(\eta)+2\eta t)=\eta.
\]
Let us denote $X_-(t)=x_-(\eta_-(t))+2\eta_-(t)t\leq X_+(t)=x_+(\eta_+(t))+2\eta_+(t)t\in\R$ the branching points at time $t$ such that $\eta_-(t)\leq 0$ and $\eta_+(t)\geq 0$. Then one can express these branches in terms of two branches. The first branch $v_0^B$ is well-defined on $[a(v_1^B),b(v_1^B)]=[X_-(t)-2\pi,X_+(t)]$, the second one $v_1^B$ is well-defined on $[a(v_0^B),b(v_0^B)]=[X_-(t),X_+(t)]$ as in Figure~\ref{fig:multivalued2}. 

\begin{figure}
\begin{center}
\includegraphics[scale=0.4]{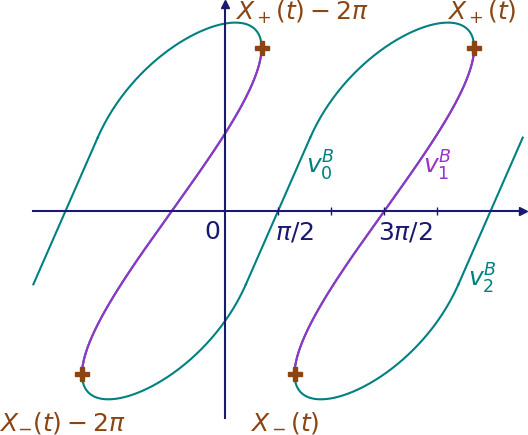}
\end{center}
\caption{Sketch of the multivalued solution of the Burgers equation obtained by the method of characteristics, with initial data $u_0(x)=-\beta\cos(x)$}\label{fig:multivalued2}
\end{figure}

In the periodic case, this leads to more branches. But as they appear two by two, the odd branches always  correspond to the branch $v_1^B$ and the even ones to the  branch $v_0^B$. As a consequence, the graphs of $u_0^B,u_2^B,\dots,u_{2P}^B$ combined are the graph of $v_0^B$ taken modulo $2\pi$ in space. The increasing part corresponds to abscissa $x=x_-(\eta)$ for some $\eta$ whereas the decreasing part corresponds to abscissa $x=x_+(\eta)$ for some $\eta$. The graphs of $u_1^B,u_3^B,\dots,u_{2P-1}^B$ combined are the graph of $v_1^B$, modulo $2\pi$, they are increasing and correspond to abscissa $x=x_+(\eta)$ for some $\eta$.
For every $\eta\in(\min(u_0),\max(u_0))$, there are exactly two antecedents $\eta=u_{n}^B(t,x_-(t,\eta))=v^B_{n\mod 2}(t,x_-(t,\eta))$ and $\eta=u_{m}^B(t,x_+(t,\eta))=v_{m\mod 2}^B(t,x_+(t,\eta))$.

\begin{prop}[Fourier coefficients of the multivalued solution to the Burgers equation]\label{prop:fourier_u_alt}
Let $u_0$ be a single well potential. Then there holds
\[
\widehat{u^B_{alt}(t)}(k)
	=-\frac{i}{2 k\pi}\int_{\min(u)}^{\max(u)}e^{-ik(x_+(\eta)+2\eta t)}-e^{-ik(x_-(\eta)+2\eta t)}\d\eta.
\]
\end{prop}

\begin{proof}
The union of graphs and the periodicity lead to
\begin{align*}
\widehat{u^B_{alt}(t)}(k)
	&=\sum_{n=0}^{2P(t,x)}(-1)^n\widehat{u_n^B}(t,x)\\
	&=\sum_{n=0}^{2P(t,x)}(-1)^n\int_{a(u_n^B)}^{b(u_n^B)}u_n^B(x)e^{-ikx}\d x\\
	&=\int_{a(v_0^B)}^{b(v_0^B)}v_0^B(x)e^{-ikx}\d x
	-\int_{a(v_1^B)}^{b(v_1^B)}v_1^B(x)e^{-ikx}\d x.
\end{align*}
The formula for single well  functions in Proposition~\ref{prop:formule_uk} becomes
\[
\widehat{v_0^B}(k)=i\frac{b(v_0^B)-a(v_0^B)}{k}-\frac{i}{2k\pi}\left(\int_{(x_+(t,\eta),\eta)\in \Gr(v_0^B)}e^{-ikx_+(t,\eta)}\d\eta
	-\int_{(x_-(t,\eta),\eta)\in \Gr(v_0^B)}e^{-ikx_-(t,\eta)}\d\eta\right).
\]
When $v=v_{1}^B$, there is only the $x_+$ part.
Therefore, the formula for increasing functions is written
\[
\widehat{v_1^B}(k)=i\frac{b(v_1^B)-a(v_1^B)}{k}+\frac{i}{2k\pi}\int_{(x_+(t,\eta),\eta)\in \Gr(v_1^B)}e^{-ikx_+(t,\eta)}\d\eta.
\]

Consequently, we have
\begin{multline*}
\widehat{u^B_{alt}(t)}(k)
	=i\frac{b(v_0^B)-a(v_0^B)}{k}
	-\frac{i}{2k\pi}\left(\int_{(x_+(t,\eta),\eta)\in \Gr(v_0^B)}e^{-ikx_+(t,\eta)}\d\eta
	-\int_{(x_-(t,\eta),\eta)\in \Gr(v_0^B)}e^{-ikx_-(t,\eta)}\d\eta\right)\\
	-i\frac{b(v_1^B)-a(v_1^B)}{k}
	-\frac{i}{2k\pi}\int_{(x_+(t,\eta),\eta)\in \Gr(v_1^B)}e^{-ikx_+(t,\eta)}\d\eta.
\end{multline*}
Now we use that the union of the graphs of $v_0^B$ and $v_1^B$ taken modulo $2\pi$ give the graph of $u_0$, moreover, $a(v_0^B)=a(v_1^B)  [\mod 2\pi]$, $b(v_0^B)=b(v_1^B) [\mod 2\pi]$, which leads to the result.
\end{proof}

We now justify why there exist admissible families for every single well initial data $u_0$.

\begin{lem}[Existence of admissible approximate initial data]\label{lem:exist_admissible}
For every single well initial data $u_0$, the set of admissible approximate initial data according to Definition~\ref{def:approximate_Birkhoff} is non empty.
\end{lem}

\begin{proof}
Let us denote
\[
I:=\int_{\min(u_0)}^{\max(u_0)}F(\eta)\d\eta.
\]
We make the following choice
\begin{enumerate}
\item (Small eigenvalues) 
If $0\leq n\varepsilon\leq I$, then we define $\lambda_n$ as the solution to
\(
\int_{-\lambda_n}^{\max(u_0)}F(\eta)\d\eta
	:=n\varepsilon.
\)
\item (Large eigenvalues) Assume that $n\varepsilon> I$, then we define 
\(
\lambda_n
	:=n\varepsilon.
\)
\item (Phase factors) We then define the approximate phase factors by the formula $\theta_0=0$ and
\begin{equation*}
\theta_{n+1}-\theta_n:=\pi-\frac{x_+(-\lambda_n)+x_-(-\lambda_n)}{2}.
\end{equation*}
\end{enumerate}
We only need to check that $\|u_0^{\varepsilon}\|_{L^2}\to \|u_0\|_{L^2}$ as $\varepsilon\to 0$. For this we use the Parseval formula~\eqref{eq:Parseval}
\begin{align*}
\frac{1}{2}\|u_0^{\varepsilon}\|_{L^2}^2
	&=\sum_{n\geq 1}\varepsilon n\gamma_n(u;\varepsilon)\\
	&=\sum_{n=1}^{\lfloor I/\varepsilon\rfloor}\varepsilon n(\lambda_{n}(u;\varepsilon)-\lambda_{n-1}(u;\varepsilon)-\varepsilon)\\
	&=\sum_{n=1}^{\lfloor I/\varepsilon\rfloor}\varepsilon\lambda_n(u;\varepsilon) -\varepsilon^2\frac{\lfloor I/\varepsilon\rfloor(\lfloor I/\varepsilon\rfloor+1)}{2}.
\end{align*}
Since we have seen that $|(\lambda_{n+1}-\lambda_n)F(-\lambda_n)-\varepsilon|\leq C(\delta)\varepsilon\sqrt{\varepsilon}$, the mesh is tending to zero and the $\eta=-\lambda_n$ are distributed by $F$. This leads to 
\[
\frac{1}{2}\lim_{\varepsilon\to 0}\|u_0^{\varepsilon}\|_{L^2}^2
	=\int_{\min(u_0)}^{\max(u_0)}\eta F(\eta)\d\eta-\frac{I^2}{2}.
\]
However, there also holds
\begin{align*}
\frac{1}{2}\|u_0\|_{L^2}^2
	&=\frac{1}{2\pi}\int_{0}^{+\infty}\eta\operatorname{Leb}(x\mid |u(x)|>\eta)\d \eta\\
	&=\int_{0}^{\max(u_0)}\eta F(\eta)\d \eta+\int_{\min(u_0)}^{0}\eta (1-F(\eta))\d\eta\\
	&=\int_{\min(u_0)}^{\max(u_0)}\eta F(\eta)\d \eta-\frac{(\min(u_0))^2}{2}.
\end{align*}
Now, we write
\begin{align*}
I=\int_0^{\max(u_0)}F(\eta)\d\eta-\int_{\min(u_0)}^{0}(1-F(\eta))\d\eta-\min(u_0).
\end{align*}
Since $u_0$ has mean zero, the first two terms cancel out, leading to $I=-\min(u_0)$. This implies the convergence property.
\end{proof}

\begin{proof}[Proof of Theorem~\ref{thm:main}]
Let $k\in\Z$ and $T>0$. Let  $\delta>0$. We have established in Theorem~\ref{thm:link_burgers} and Proposition~\ref{prop:fourier_u_alt} that
\[
\limsup_{\varepsilon\to 0}\sup_{t\in[0,T]}\left|\widehat{u^{\varepsilon}(t)}(k)
	-\widehat{u_{alt}^B(t)}(k)\right|
		\leq C\delta.
\]
We now pass to the limit $\delta\to 0$ and deduce that uniformly on $[0,T]$, there holds $\widehat{u^{\varepsilon}}(k)\to \widehat{u^B_{alt}}(k)$. Moreover, by conservation of the $L^2$ norm and the assumption $\|u_0^{\varepsilon}\|_{L^2}\to\|u_0\|_{L^2}$, all the solutions $u^{\varepsilon}$ are bounded in $L^2$. This is enough to conclude to the weak convergence of $u^{\varepsilon}$ to $u^{B}_{alt}$ in $L^2(\T)$.
\end{proof}

\begin{rk}(Third order Benjamin-Ono equation)\label{rk:hierarchy}
Let us now consider the third order equation in the Benjamin-Ono hierarchy
\[
\partial_t u=\partial_x\left(-\varepsilon^2\partial_{xx}u-\frac{3}{2}\varepsilon u|D|u-\frac{3}{2}\varepsilon H(u\partial_x u)+u^3\right),
\]
where $H$ is the Hilbert transform, i.e. the Fourier multiplier by $-i\operatorname{sgn}(n)$. The spectral parameters are the same as for~\eqref{eq:bo}, and in the time evolution, by considering the rescaled function
\[
v(t)=\frac{1}{\varepsilon}u\left(\frac{t}{\varepsilon^2},x\right),
\]
one can see that the frequencies from~\cite{bo_hierarchie}
\[
\omega_n^{(3)}(v;1)=n^3+n\sum_{p\geq 1}p\gamma_p(v;1) -3\sum_{p\geq 1}\min(p,n)^2\gamma_p(v;1)+3\sum_{p,q\geq 1}\min(p,q,n)\gamma_p(v;1)\gamma_q(v;1),
\]
should simply be replaced by
\begin{align*}
\omega_n^{(3)}(u;\varepsilon)
	&=\varepsilon^2 \omega_n^{(3)}(u/\varepsilon;1)\\
	&=\varepsilon^2 n^3+\varepsilon n\sum_{p\geq 1}p\gamma_p(u;\varepsilon) -3\varepsilon\sum_{p\geq 1}\min(p,n)^2\gamma_p(u;\varepsilon)+3\sum_{p,q\geq 1} \min(p,q,n)\gamma_p(u;\varepsilon)\gamma_q(u;\varepsilon).
\end{align*}
As a consequence, the formula $\lambda_n(u;\varepsilon)=\varepsilon n-\sum_{p\geq n+1}\gamma_p$ and the Parseval formula~\eqref{eq:Parseval} lead to
\begin{multline*}
\omega_{n+1}^{(3)}(u;\varepsilon)-\omega_n^{(3)}(u;\varepsilon)\\
	=3\varepsilon^2 n^2+3\varepsilon^2 n+\varepsilon^2+\varepsilon \sum_{p\geq 1}p\gamma_p(u;\varepsilon)-3\varepsilon\sum_{p\geq n+1}(2n+1)\gamma_p(u;\varepsilon)+3\sum_{p,q\geq n+1}\gamma_p(u;\varepsilon)\gamma_q(u;\varepsilon),
\end{multline*}
and finally
\begin{equation*}
\omega_{n+1}^{(3)}(u;\varepsilon)-\omega_n^{(3)}(u;\varepsilon)
	=3\lambda_n^2+3\varepsilon\lambda_n+\varepsilon^2+\|u\|_{L^2(\T)}^2/2.
\end{equation*}
As $\varepsilon\to 0$, an adaptation of the proof of Theorem~\ref{thm:link_burgers} would lead to the convergence of $u^{\varepsilon}$ to the solution $\widetilde{u}^B$ to the third equation in the inviscid Burgers hierarchy
\[
\widetilde{u}^B (x,t)=u_0\left(x-\left(3 \widetilde{u}^B(x,t)^2+\|u_0\|_{L^2}^2/2\right)t\right),
\]
which is comparable to the behavior on the line~\cite{MillerXu2011hierarchy}.
\end{rk}

\begin{lem}[Weak convergence after the breaking time]\label{lem:CV_weak}
Let $u_0(x)=-\cos(x)$. For $t$ right after the breaking time for the Benjamin-Ono equation, there holds
\[
\|u_0\|_{L^2(\T)}>\|u^B_{alt}(t)\|_{L^2(\T)}.
\]
As a consequence, the solution $u^{\varepsilon}$ to equation~\eqref{eq:bo} with initial data $u_0$
\[
\|u^{\varepsilon}(t)\|_{L^2(\T)}= \|u_0\|_{L^2(\T)}>\|u^B_{alt}(t)\|_{L^2(\T)},
\]
and the convergence of $u^{\varepsilon}(t)$ to $u^B_{alt}(t)$ cannot be strong in $L^2(\T)$.
\end{lem}

\begin{proof}
We have $x_-(\eta,t)=\pi-\arccos(\eta)+2\eta t$ and $x_+(\eta,t)=\pi+\arccos(\eta)+2\eta t$. We first study $x_+$ as a function of $\eta$. We get
\[
\partial_{\eta}x_+(\eta,t)=-\frac{1}{\sqrt{1-\eta^2}}+2t.
\]
When $2t>1$, there is exactly one solution $\eta_+=\sqrt{1-1/2t}$ in $(0,1]$ and one solution $\eta_-=-\eta_+$ in $[-1,0)$, and $\partial_{\eta}$ is negative on $[\eta_-,\eta_+]$. As a consequence, $x_+$ is increasing on $[-1,\eta_-]$ and $[\eta_+,1]$, and decreasing on $[\eta_-,\eta_+]$. Moreover, one can see that $\partial_{\eta}x_+$ is an even function of $\eta$ and $x_+(0)=\pi/2$, so that $x_+(\eta)-\pi/2$ is an odd function of $\eta$.

Consequently, right after the breaking time, one can describe the branches of one period of $u^B(t,x)$ as follows (the notation is similar to Figure~\ref{fig:multivalued2} except that we have split $v_0^B$ into a left part denoted $v_0^B$ and a right part denoted $v_2^B$ by restricting first the solution to $[0,2\pi]$). Let us denote $X_{+}=x_{+}(\eta_{+})$ and $X_-=x_+(\eta_-)$. There is one upper branch $v_0^B$ from $(0, u^B(t,0))$ to $(X_+,\eta_+)$, one middle branch $v_1^B$ from $(X_-,\eta_-)$ to $(X_+,\eta_+)$ and one lower branch $v_2^B$ from $(X_-,\eta_-)$ to $(2\pi,u^B(t,2\pi))$. Moreover, $v_0^B$ is increasing from $(0,u^B(t,0))$ to $(\pi/2+2t,1)$, and decreasing from $(\pi/2+2t,1)$ to $(X_+,\eta_+)$. Similarly, $v_2^B$ is decreasing from $(X_-,\eta_-)$ to $(2\pi-2t, -1)$ and increasing from $(2\pi-2t, -1)$ to $(2\pi,u^B(t,2\pi))$.

We note that $X_+=\pi+\arccos(\eta_+)+2\eta_+ t\to \frac{3\pi}{2}^+$ as $2t\to 1^+$, $2\pi-1>\frac{3\pi}{2}$ and similarly,   $X_-=\pi+\arccos(\eta_-)+2\eta_- t\to \frac{3\pi}{2}^-$ as $2t\to 1^+$ and $\frac{\pi}{2}+1<\frac{3\pi}{2}$, so that when $2t-1$ is small and positive,
\[
\frac{\pi}{2}+2t<X_-<\frac{3\pi}{2}<X_+<2\pi-2t.
\]
For $x\in(\frac{3\pi}{2},X_+)$, then $v_0^B(x)>v_1^B(x)>0>v_2^B(x)$, and as a consequence,
\[
u^B_{alt}(x)=v_2^B(x)-v_1^B(x)+v_0^B(x)>v_0^B(x).
\]
Moreover, $v_2^B$ and $v_0^B$ are decreasing whereas $v_1^B$ is increasing, so that
\[
u^{B}_{alt}(x)\leq v_2^B\left(\frac{3\pi}{2}\right)-v_1^B\left(\frac{3\pi}{2}\right)+v_0^B\left(\frac{3\pi}{2}\right)=0.
\]
We conclude that when $\eta\in [0,\eta_+]$,
\[
F(\eta)>\frac{1}{2\pi}\operatorname{Leb}(x\in[0,2\pi]\mid u^B_{alt}(x)>\eta).
\]
Similarly, 
\[
F(\eta)>\frac{1}{2\pi}\operatorname{Leb}(x\in[0,2\pi]\mid u^B_{alt}(x)<-\eta).
\]
Finally, we write
\begin{align*}
\frac{1}{2}\|u^B_{alt}\|_{L^2}^2
	&=\frac{1}{2\pi}\int_{0}^{\infty}\lambda \operatorname{Leb}(x\in[0,2\pi]\mid |u^B_{alt}(x)|>\lambda)\d\lambda\\
	&=\frac{1}{2\pi}\int_{0}^{\infty}\eta \operatorname{Leb}(x\in[0,2\pi]\mid u^B_{alt}(x)>\eta)\d\eta
	+\frac{1}{2\pi}\int_{0}^{\infty}\eta \operatorname{Leb}(x\in[0,2\pi]\mid u^B_{alt}(x)<-\eta)\d\eta.
\end{align*}
By splitting the integrals between the zones $\eta> \eta_+$ and $0\leq \eta\leq \eta_+$, we conclude that
\[
\frac{1}{2}\|u^B_{alt}\|_{L^2}^2
	<\int_0^{\infty}\eta F(\eta)\d\eta
	=\frac{1}{2}\|u_0\|_{L^2}^2.\qedhere
\]
\end{proof}

\section{Lax eigenvalues for initial data \texorpdfstring{$u_0(x)=-\beta\cos(x)$}{-beta cos(x)}}\label{part:Lax}

The aim of this part is to establish the asymptotic expansion on the Lax eigenvalues from Theorem~\ref{thm:asymptotics}. Using the fact that any eigenfunction $f_n(u_0;\varepsilon)\in L^2_+(\T)$ for the Lax operator admits an analytic expansion on the complex unit disc, we derive an integral identity in part~\ref{subsection:eigenvalue_equation}. Then we apply the method of stationary phase and the Laplace method to deduce an asymptotic expansion of the Lax eigenvalues in part~\ref{subsection:stationary_Laplace}.
Conversely, we  justify that this method enables us to list all the Lax eigenvalues in the two regions $\Lambda_+(\delta)$ and $\Lambda_-(\delta)$ in part~\ref{subsection:characterization}.

\subsection{Eigenvalue equation}\label{subsection:eigenvalue_equation}


In this part, we establish the eigenvalue equation for $u(x)=-\beta\cos(x)$.

\begin{prop}[Eigenvalue equation]\label{prop:eigenvalue_equation}
Let $u(x)=-\beta\cos(x)$. Let $\lambda(u;\varepsilon)$ be an eigenvalue of $L_u(\varepsilon)$. Then
\begin{equation}\label{eq:eigenvalues}
\int_0^{\pi}\cos\left(\frac{\beta}{\varepsilon}\sin\varphi+\left(1+\frac{\lambda}{\varepsilon}\right)\varphi\right)\d\varphi
	=\sin\left(\pi\frac{\lambda}{\varepsilon}\right)\int_0^{\infty}\exp\left(-\frac{\beta}{\varepsilon}\sinh(x)+\left(1+\frac{\lambda}{\varepsilon}\right)x\right)\d x.
\end{equation}
\end{prop}

A possible further generalization for general trigonometric polynomials, although more technical, would allow us to use the comonotone approximation theorems of continuous functions by trigonometric polynomials (see \cite{LorentzZeller1968} in the case of single well potentials and \cite{DzyubenkoPleshakov2008} in more general cases), however, we were not able to push further this approach.


\begin{proof}
A scaling argument implies that for the Lax operator $L_u(\varepsilon)$ associated to the Benjamin-Ono equation with dispersion parameter $\varepsilon$, we have $\lambda_n(u;\varepsilon)=\varepsilon\lambda_n(\frac{u}{\varepsilon};1)$. It is therefore enough to only tackle the case $\varepsilon=1$ and replace later $\beta$ by $\beta/\varepsilon$.

Let $f$ be an eigenvector of the Lax operator $L_u=D-T_u$ with eigenvalue $\lambda$. Let $\alpha=-\beta/2$. Since $u(x+\pi)=-u(x)$, the spectrum is unchanged when $\beta$ becomes $-\beta$ and we rather study $u(x)=2\alpha\cos(x)$. We expand $f$ and $\Pi u $ as holomorphic functions on $\D$, and 
\[
\overline{\Pi u}(z)=\alpha z^{-1}
\]
as a holomorphic function on $\C\setminus\{0\}$.
 Then the Szeg\H{o} projector has the expression
\[
\Pi(\overline{\Pi u} f)
	=\frac{1}{2\pi i}\oint_{\partial\D}\frac{f(\zeta)}{\zeta-z}\alpha\zeta^{-1}\d\zeta.
\]
Applying the residue formula, we get for $z\in\D$ that
\begin{align*}
\Pi(\overline{\Pi u} f)(z)
	&=\overline{\Pi u}f(z)+\alpha\Res_{\zeta=0}\left(\frac{f(\zeta)}{\zeta-z}\zeta^{-1}\right)\\
	&=\overline{\Pi u}f(z)-\alpha\frac{f(0)}{z}.
\end{align*}
The equation $Df-\Pi(uf)=\lambda f$ satisfied by $f$ becomes
\begin{equation}\label{eq:f_res}
zf'(z)-\left(\alpha z+{\alpha}z^{-1}+\lambda\right)f(z)=-\alpha f(0) z^{-1}.
\end{equation}
Since $f$ is holomorphic, the right hand side does not go to zero as $z\to0$. As a consequence, we have $f(0)\neq 0$, and we can therefore assume that $f(0)=1$.

Let us choose the branch of the logarithm corresponding to $\arg(z)\in(0,2\pi)$ and define
\[
h(z):=z^{-\lambda}\exp\left(-\alpha z+\alpha z^{-1}\right).
\]
Then $h$ is solution on $\C\setminus\R_+$ to
\[
zh'(z)=-\left(\alpha z+{\alpha}z^{-1}+\lambda\right)h(z).
\]
We deduce that equation~\eqref{eq:f_res} is equivalent to
\begin{equation}\label{eq:f_identity}
(fh)'(z)=-\alpha z^{-\lambda-2}\exp\left(-\alpha z+\alpha z^{-1}\right)=-\alpha z^{-2}h(z).
\end{equation}
By assumption, we have $\alpha>0$. For $z\in\C\setminus\R_+$, we choose a path $\gamma_z$ joining $0$ to $z$ in $\C\setminus\R_+$ and such that $\gamma_z(t)=-t$ if $t\in[0,t_0]$ for some $t_0>0$. Since $f$ is holomorphic on $\D$, we get
\begin{equation}\label{eq:f_line}
f(z)
	=-\alpha z^{\lambda}\exp\left(\alpha z-\alpha z^{-1}\right)
		\int_{\gamma_z}\zeta^{-\lambda-2}\exp\left(-\alpha\zeta+\alpha\zeta^{-1}\right)\d\zeta,
\end{equation}
and the integral is absolutely convergent.

We first prove that this expression defines a holomorphic function on $\C\setminus\R_+$ satisfying the eigenfunction equation, and converging to $1$ as $z\to 0$. Indeed, a Taylor expansion of $f$ around $0$ of the form $S_n(z)=1+a_1z+\dots +a_nz^n$ with $a_0=1$ and $n\geq 1$ transforms equation~\eqref{eq:f_identity} into
\begin{multline*}
((f-S_n)h)'(z)
	=-\alpha z^{-2}h(z)
	-(a_1+\dots+na_{n}z^{n-1})h(z)\\
	+(1+a_1z+\dots+a_nz^n)(\alpha +\alpha z^{-2}+\lambda z^{-1})h(z).
\end{multline*}
We now define the coefficients $a_k$ by induction using this formula in order to cancel all the negative powers of $z$. We note that the coefficient before $z^{-2}h(z)$ is $0$, and the coefficient before $z^{-1}h(z)$ is $0$ if $\alpha a_1+\lambda=0$. Next, the coefficient in front of $z^{j}h(z)$, $j\geq 0$, is
\[
a_{j+1}-\alpha a_j+\alpha a_{j+2}+\lambda a_{j+1},
\]
one can therefore choose $a_{j+2}$ in order to cancel this term. As a consequence, at rank $n$, one can cancel the terms up to $z^{n-2}h(z)$. We end up with
\begin{equation*}
((f-S_n)h)'(z)
	=R_n(z)z^{n-1}h(z),
\end{equation*}
where $R_n=(\lambda-n)a_n+\alpha a_{n-1}+\alpha a_n z$ is a holomorphic remainder term. We conclude that
\begin{equation*}
f(z)-S_n(z)	=h(z)^{-1}\int_{\gamma_z}R_n(\zeta)\zeta^{n-1}h(\zeta)\d\zeta,
\end{equation*}
with $h(\zeta)=\zeta^{-\lambda}\exp\left(-\alpha\zeta+\alpha\zeta^{-1}\right)$. Choosing $n$ large with respect to $\lambda$, we see that this defines a holomorphic function on $\C\setminus\R_+$, satisfying the differential ODE~\eqref{eq:f_identity} and converging to $1$ as $z\to 0$. Moreover, for every $r>0$, the limits $f(r+iy)$ and $f(r-iy)$ exist as $y\to 0^+$, therefore $f$ is an eigenfunction associated to $\lambda$ if and only if for every $r>0$, we have 
\begin{equation}\label{eq:f_compatibility}
\lim_{y\to 0^+}f(r+iy)=\lim_{y\to 0^+}f(r-iy). 
\end{equation}
We now check this property.

We first assume that $\lambda$ is not an integer. For $r>0$, we integrate \eqref{eq:f_identity} over the circle centered at $0$ of radius $r$, starting at $r+i0$ and ending at $r-i0$ in the counterclockwise direction. This leads to a second condition
\begin{equation*}
(e^{-2i\pi\lambda}-1)f(r)\exp\left(-\alpha r+\alpha r^{-1}\right)\\
	=-\alpha \int_0^{2\pi}
	\exp\left(-\alpha r e^{i\theta}+
	\alpha r^{-1}e^{-i\theta}\right)
	r^{-1}e^{-i(\lambda+1)\theta}i\d\theta.
\end{equation*}
By analytic continuation, we obtain that for every $z\in\C\setminus\{0\}$, we have
\begin{equation}\label{eq:f_circle}
(e^{-2i\pi\lambda}-1)f(z)\exp\left(-\alpha z+\alpha z^{-1}\right)\\
	=-i\alpha z^{-1}\int_0^{2\pi}
	\exp\left(-\alpha z e^{i\theta}+
	\alpha z^{-1}e^{-i\theta}\right)
	e^{-i(\lambda+1)\theta}\d\theta.
\end{equation}
This expression defines a holomorphic function $f$ outside the origin, solving the ODE~\eqref{eq:f_res} of order one. Therefore, $f$ defines an eigenfunction with eigenvalue $\lambda$ if and only if this expression coincides with~\eqref{eq:f_line} at one non singular point, for instance at the point $z=-1$:
\begin{equation*}
(e^{-2i\pi\lambda}-1)\int_{0}^1 t^{-\lambda-2}\exp\left(\alpha t-\alpha t^{-1}\right) \d t\\
	=i\alpha \int_0^{2\pi}
	\exp\left(\alpha  e^{i\theta}-
	\alpha e^{-i\theta}\right)
	e^{-i(\lambda+1)\theta}\d\theta,
\end{equation*}
or
\begin{equation*}
e^{-i\pi\lambda}2\sin(-\pi\lambda)\int_{0}^1 t^{-\lambda-2}\exp\left(\alpha t-\alpha t^{-1}\right) \d t\\
	=\alpha \int_0^{2\pi}
	\exp\left(i(2	\alpha \sin(\theta)-(\lambda+1)\theta)\right)\d\theta.
\end{equation*}
We set the change of variable $t=e^{-x}$ and $\theta=\varphi+\pi$, and get that this is equivalent to
\begin{equation*}
\sin(\pi\lambda)\int_0^{\infty}\exp\left({(\lambda+1)x-2\alpha\sinh(x)}\right)\d x\\
	=\int_0^{\pi}\cos\left((\lambda+1)\varphi+2\alpha \sin(\varphi)\right)\d\varphi.
\end{equation*}

When $\lambda$ is an integer, the condition~\eqref{eq:f_compatibility} is satisfied if and only if the function $f$ given by~\eqref{eq:f_line} is holomorphic:
\begin{equation*}
0
	=\int_{C(0,r)}\zeta^{-\lambda-2}\exp\left(-\alpha\zeta+\alpha \zeta^{-1}\right)\d\zeta,
\end{equation*}
where $C(0,r)$ is the circle of radius $r$ centered at $0$. Since the integrand is holomorphic outside the origin, this integral does not depend on $r$ so it is enough to calculate it for $r=1$:
\[
0
	=\int_{0}^{2\pi}e^{i\theta(-\lambda-1)}\exp\left(-\alpha e^{i\theta}+\alpha e^{-i\theta}\right)\d\theta,
\]
and this also leads to the result.
\end{proof}

%

\subsection{Asymptotic expansion for the Lax eigenvalues}\label{subsection:stationary_Laplace}

In this part, we apply the stationary phase and Laplace methods into identity~\eqref{eq:eigenvalues} and get an asymptotic expansion of the Lax eigenvalues $\lambda$.

In order to get a uniform bound on the remainder terms, we need do ensure that the stationary point of the phase remains sufficiently far from the integral boundaries. Therefore, we fix a small parameter $\delta>0$, and we only consider the eigenvalues $\lambda$ such that $\nu=\lambda+\varepsilon$ is  inside one of the two intervals $\Lambda_-(\delta)=[-\beta+\delta,\beta-\delta]$ and $\Lambda_+(\delta)=[\beta+\delta,+\infty)$.

We apply the method of stationary phase for the first term
\[
I_1(\varepsilon,\nu):=\int_0^{\pi}\exp\left(i\left(\frac{\beta}{\varepsilon}\sin\varphi+\frac{\nu}{\varepsilon}\varphi\right)\right)\d\varphi
\]
and the Laplace method for the second term
\[
I_2(\varepsilon,\nu):=\int_0^{\infty}\exp\left(-\frac{\beta}{\varepsilon}\sinh(x)+\frac{\nu}{\varepsilon}x\right)\d x
\]
that appear in the identity~\eqref{eq:eigenvalues}
\(
\Re(I_1(\varepsilon,\nu))=\sin\left(\pi\frac{\lambda}{\varepsilon}\right)I_2(\varepsilon,\nu),
\)
which we write
\begin{equation}\label{eq:identity}
I(\varepsilon,\nu)=0
\end{equation}
with $\nu=\lambda+\varepsilon$ and
\begin{equation*}
I(\varepsilon,\nu)
	=\Re(I_1(\varepsilon,\nu))+\sin\left(\pi\frac{\nu}{\varepsilon}\right)I_2(\varepsilon,\nu).
\end{equation*}

To estimate the large eigenvalues, we choose $K(\delta)>0$  such that $\frac{1}{K(\delta)}\|u\|_{L^2}^2< 2\delta$, in order to get thanks to the Parseval formula~\eqref{eq:Parseval}
\[
\sum_{k\geq K(\delta)/\varepsilon}\gamma_k(u;\varepsilon)
	\leq \frac{\varepsilon}{K(\delta)}\sum_{k\geq K(\delta)/\varepsilon} k\gamma_k(u;\varepsilon)
	\leq \frac{1}{2K(\delta)}\|u\|_{L^2}^2
	< \delta.
\]
As a consequence, since 
\(
\lambda_n(u;\varepsilon)=n\varepsilon-\sum_{k\geq n+1}\gamma_k(u;\varepsilon),
\)
 one can see that for $n\geq K(\delta)/\varepsilon$, there holds $\lambda_n+\varepsilon\geq K(\delta)-\delta$.

\paragraph{Method of stationary phase for the first term} Let us start with $I_1(\varepsilon,\nu)=\int_0^{\pi}e^{iS_1(x,\nu)/\varepsilon}\d x$, where the phase is equal to
\[
S_1(x,\nu)=\beta\sin(x)+\nu x=F_1(x)+\nu x,
\]
$F_1(x)=\beta\sin(x)$.
Since $\partial_x S_1(x,\nu)=\beta\cos(x)+\nu$, we have the following alternative.
\begin{enumerate}
\item (Small eigenvalues) If  $|\nu|\leq \beta-\delta$, there is a unique critical point
\[
x_1(\nu)=\arccos(-\nu/\beta).
\]
Moreover, there exists $\delta_1(\delta)>0$ such that
\[
x_1(\nu)\in [\arccos(1-\delta/\beta),\arccos(-1+\delta/\beta)]\subset[2\delta_1(\delta),\pi-2\delta_1(\delta)].
\]
Therefore, there exists $\delta_2(\delta)>0$ such that $F''_1(x)=-\sin(x)\leq -\delta_2(\delta)$ for $x\in [\delta_1(\delta),\pi-\delta_1(\delta)]$. 
We also compute 
\[S_1(x_1(\nu),\nu)=\sqrt{\beta^2-\nu^2}+\nu x_1(\nu)=\sqrt{\beta^2-\nu^2}+\nu \arccos(-\nu/\beta),
\]
\[\partial_{xx}S_1(x_1(\nu),\nu)=-F_1''(x_1(\nu))=-\sqrt{\beta^2-\nu^2}.
\] 
 
 From the stationary phase method (see for instance~\cite{Hormander2015}, chapter 7, or adapt directly the proof of~\cite{Faraut2021}), there exists $C(\delta)>0$ such that for every  $|\nu|\leq \beta-\delta$, there holds
\begin{equation*}
\left|I_1(\varepsilon,\nu)-\frac{\sqrt{2\pi \varepsilon }}{\sqrt{\beta^2-\nu^2}}\exp\left(i\left(\frac{S_1(x_1(\nu),\nu)}{\varepsilon}-\frac{\pi}{4}\right)\right)\right|
	\leq C(\delta)\varepsilon,
\end{equation*}
implying
\begin{equation}\label{eq:I1-}
\left| \Re \left(I_1(\varepsilon,\nu)\right)-\frac{\sqrt{2\pi \varepsilon }}{\sqrt{\beta^2-\nu^2}}\cos\left(\frac{S_1(x_1(\nu),\nu)}{\varepsilon}-\frac{\pi}{4}\right)\right|
	\leq C(\delta)\varepsilon.
\end{equation}

\item (Large eigenvalues) If $K(\delta)\geq \nu\geq\beta+\delta$, the phase has no critical point since  $\partial_xS_1(0,\nu)=\beta\cos(x)+\nu\geq\delta$ for every $x$. We know that $S_1(0,\nu)=0$, therefore, there exists $C(\delta)>0$ such that for every  $\nu/\beta\geq 1+\delta/\beta$,
\begin{equation}\label{eq:I1+}
|I_1(\varepsilon,\nu)|
	\leq C(\delta)\varepsilon.
\end{equation}
\end{enumerate}

\paragraph{Method of Laplace for the second term}

Let us now analyze $I_2(\varepsilon)=\int_0^{+\infty}e^{S_2(x,\nu)/\varepsilon}\d x$, where the phase is equal to
\[
S_2(x,\nu)=-\beta\sinh(x)+\nu x=F_2(x)+\nu x,
\]
$F_2(x)=-\beta\sinh(x)$.
Since $\partial_xS_2(x,\nu)=-\beta\cosh(x)+\nu$, the following holds.
\begin{enumerate}
\item (Small eigenvalues) If  $|\nu|\leq \beta-\delta$, we have $\partial_xS_2(x,\nu)\leq-\beta+\nu\leq-\delta$ for every $x$, therefore there is no critical point. We get from the Laplace method that
\begin{equation}\label{I2-}
|I_2(\varepsilon,\nu)|\leq C(\delta)\varepsilon.
\end{equation}

\item (Large eigenvalues) If $K(\delta)\geq \nu\geq \beta+\delta$, then there is a unique critical point
\[
x_2(\nu)=\cosh^{-1}\left(\frac{\nu}{\beta}\right).
\]
There exists $\delta_1(\delta)>0$ such that $x_2(\nu)\geq 2\delta_1(\delta)$ for every $\nu\geq \beta+\delta$, and there exists $\delta_2(\delta)>0$ such that $F''_2(x)=-\beta\sinh(x)\leq-\delta_2(\delta)$ for every $x\geq \delta_1(\delta)$. Moreover, we have
\[
S_2(x_2(\nu),\nu)=-\sqrt{\nu^2-\beta^2}+\nu x_{2}(\nu),\]
\[F''_2(x_{2}(\nu))=-\sqrt{\nu^2-\beta^2}.
\]
One can therefore apply the adapted Laplace method and get
\begin{equation}\label{eq:I2}
\left|I_2(\varepsilon,\nu)
	-\frac{\sqrt{2\pi\varepsilon}}{\sqrt{\nu^2-\beta^2}}e^{S_2(x_2(\nu),\nu)/\varepsilon}\right|
	\leq C(\delta)\varepsilon.
\end{equation}
\end{enumerate}

\paragraph{Small eigenvalues}
For every $|\nu|\leq \beta-\delta$, we conclude that
\begin{equation}\label{eq:approxI-}
\left|I(\varepsilon,\nu)-\frac{\sqrt{2\pi\varepsilon}}{\sqrt{\beta^2-\nu^2}}\cos\left(\frac{S_1(x_1(\nu),\nu)}{\varepsilon}-\frac{\pi}{4}\right)\right|
	\leq C(\delta)\varepsilon.
\end{equation}
Therefore, identity~\eqref{eq:identity}  implies that given a small Lax eigenvalue $\lambda$, then $\nu=\lambda+\varepsilon$ satisfies
\[
\left|\frac{\sqrt{2\pi\varepsilon}}{\sqrt{\beta^2-\nu^2}}\cos\left(\frac{S_1(x_1(\nu),\nu)}{\varepsilon}-\frac{\pi}{4}\right)\right|
	\leq C(\delta)\varepsilon.
\]
Taking the limit $\varepsilon\to 0$, we conclude that $\cos\left(\frac{S_1(x_1(\nu),\nu)}{\varepsilon}-\frac{\pi}{4}\right)$ should be close to~$0$ as soon as  $\varepsilon<\varepsilon_0(\delta)$ is small enough. Therefore there exists an integer $N(\varepsilon,\nu)$ such that
\[
\left|\frac{S_1(x_1(\nu),\nu)}{\varepsilon}-\frac{3\pi}{4}
	-\pi N(\varepsilon,\nu)\right|
	\leq C(\delta)\sqrt{\varepsilon}.
\]
Let us recall the definition of
\[
F(\eta)=\frac{1}{2\pi}\operatorname{Leb}\{x\in [0,2\pi]\mid u(x)\geq \eta\}.
\]
Then one can see that whenever $-\beta<\eta<\beta$, we have
\(
2\pi F(\eta)=2\arccos\left(\frac{\eta}{\beta}\right).
\)
As a consequence,
\begin{multline*}
2\pi\int_{-\nu}^{\beta}F(\eta)\d\eta
	=2\beta\left[x\arccos(x)-\sqrt{1-x^2}\right]_{-\nu/\beta}^1\\
	=2\left(\nu\arccos\left(-\frac{\nu}{\beta}\right)+\sqrt{\beta^2-\nu^2}\right)
	=2S_1(x_1(\nu),\nu).
\end{multline*}
We have therefore proven that
\begin{equation}\label{ineq:defN-}
\left|\int_{-\nu}^{\beta}F(\eta)\d\eta
-\frac{3\varepsilon}{4}
	-\varepsilon N(\varepsilon,\nu)\right|
	\leq C(\delta)\varepsilon\sqrt{\varepsilon}.
\end{equation}
The integral term is bounded below since
\(\int_{\beta-\delta}^{\beta}F(\eta)\d\eta\geq\frac{1}{C(\delta)}.
\) 
We deduce that when $\varepsilon<\varepsilon_0(\delta)$ is small enough, we have that necessarily $N(\varepsilon,\nu)\geq 0$ (and even $N(\varepsilon,\nu)\geq \frac{1}{C(\delta)\varepsilon}$).

\paragraph{Large eigenvalues}
In the case $K(\delta)\geq \nu\geq\beta+\delta$, we have proven that
\begin{equation}\label{eq:approxI+}
\left|I(\varepsilon,\nu)-\sin\left(\pi\frac{\nu}{\varepsilon}\right)\frac{\sqrt{2\pi\varepsilon}}{\sqrt{\nu^2-\beta^2}}e^{S_2(x_2(\nu),\nu)/\varepsilon}\right|
	\leq C(\delta)\varepsilon.
\end{equation}
Then identity~\eqref{eq:identity} implies that given a large Lax eigenvalue $\lambda$, then $\nu=\lambda+\varepsilon$ satisfies
\[
\left|\sin\left(\pi\frac{\nu}{\varepsilon}\right)\frac{\sqrt{2\pi\varepsilon}}{\sqrt{\nu^2-\beta^2}}e^{S_2(x_2(\nu),\nu)/\varepsilon}\right|
	\leq C(\delta)\varepsilon.
\]
We introduce the function $x\mapsto \psi(x):=-\sqrt{x^2-1}+x\cosh^{-1}(x)$ on $[1,\infty)$. This function satisfies $\psi(1)=0$ and its derivative on $(1,\infty)$ is
\[
\psi'(x)=-\frac{x}{\sqrt{x^2-1}}+\cosh^{-1}(x)+\frac{x}{\sqrt{x^2-1}}\geq 0.
\]
This implies that $\psi(x)\geq 0$ on $[1,\infty)$. Since $\nu> \beta$, then $S_2(x_2(\nu),\nu)=\beta \psi(\nu/\beta)\geq 0$. We deduce that
\[
\left|\sin\left(\pi\frac{\nu}{\varepsilon}\right)\right|
	\leq \left| \sin\left(\pi\frac{\nu}{\varepsilon}\right)e^{S_2(x_2(\nu),\nu)/\varepsilon}\right|
	\leq C(\delta)\sqrt{\varepsilon}.
\]
As a consequence, $\sin(\pi\nu/\varepsilon)$ is close to $0$ for small $\varepsilon$, and we have the more precise asymptotics
\begin{equation}\label{eq:defN+}
\left|\nu-(1+N(\varepsilon,\nu))\varepsilon\right|
	\leq C(\delta)\varepsilon\sqrt{\varepsilon}.
\end{equation}
Since $\nu\geq \beta+\delta$, we know that necessarily, $N(\varepsilon,\nu)\geq 0$ (and even $N(\varepsilon,\nu)\geq\frac{\beta}{\varepsilon}$ for $\varepsilon<\varepsilon_0(\delta)$).

\subsection{Characterization of the eigenvalues}\label{subsection:characterization}
Conversely, we prove in this part that the  asymptotic expansions obtained in part~\ref{subsection:stationary_Laplace} actually correspond to eigenvalues for the Lax operator, and therefore we conclude the proof of Theorem~\ref{thm:asymptotics}.

More precisely, in each of the two regimes $\nu\in\Lambda_-(\delta)=[-\beta+\delta,\beta-\delta]$ and $\nu\in\Lambda_+(\delta)=[\beta+\delta,\infty)$, we establish that there is exactly one eigenvalue $\lambda$ such that $\nu=\lambda+\varepsilon$ satisfies $N(\varepsilon,\nu)=N$, as soon as $N$ is compatible with the conditions $|\nu|\leq \beta-\delta$ or $\nu\geq\beta+\delta$.

In this purpose, we fix $\varepsilon$ small enough and study the variations of the two functions of $\nu$ defined as
 \(
I_1(\varepsilon,\nu)=\int_0^{\pi}\exp\left(i\left(\frac{\beta}{\varepsilon}\sin( x)+\frac{\nu}{\varepsilon} x\right)\right)\d x
\)
and
\(
I_2(\varepsilon,\nu)=\int_0^{\infty}\exp\left(-\frac{\beta}{\varepsilon}\sinh(x)+\frac{\nu}{\varepsilon}x\right)\d x.
\)
We have
 \[
\partial_{\nu}I_1(\varepsilon,\nu)
	=\frac{i}{\varepsilon}\int_0^{\pi}x\exp\left(i\left(\frac{\beta}{\varepsilon}\sin(x)+\frac{\nu}{\varepsilon}x\right)\right)\d x
\]
and
\[
\partial_{\nu}I_2(\varepsilon,\nu)
	=\frac{1}{\varepsilon}\int_0^{\infty}x\exp\left(-\frac{\beta}{\varepsilon}\sinh(x)+\frac{\nu}{\varepsilon}x\right)\d x.
\]
We also recall that
\[
I(\varepsilon,\nu)=\Re(I_1(\varepsilon,\nu))+\sin\left(\pi\frac{\nu}{\varepsilon}\right)I_2(\varepsilon,\nu).
\]

\paragraph{Small eigenvalues}
First, we assume that $|\nu|\leq \beta-\delta$. Then the stationary phase method implies that
\[
\left|\partial_{\nu}I_1(\varepsilon,\nu)-i\frac{x_1(\nu)}{\varepsilon}\frac{\sqrt{2\pi \varepsilon }}{\sqrt{\beta^2-\nu^2}}\exp\left(i\left(\frac{S_1(x_1(\nu),\nu)}{\varepsilon}-\frac{\pi}{4}\right)\right)\right|
	\leq C(\delta),
\]
so that
\[
\left|\partial_{\nu}\Re(I_1(\varepsilon,\nu))+\frac{x_1(\nu)}{\sqrt{\varepsilon}}\frac{\sqrt{2\pi  }}{\sqrt{\beta^2-\nu^2}}\sin\left(\frac{S_1(x_1(\nu),\nu)}{\varepsilon}-\frac{\pi}{4}\right)\right|
	\leq C(\delta),
\]
whereas
\[
\left|\partial_{\nu} I_2(\varepsilon,\nu)\right|
	\leq C(\delta).
\]
We conclude that
\[
\left|\partial_{\nu}I(\varepsilon,\nu)
	+\frac{x_1(\nu)}{\sqrt{\varepsilon}}\frac{\sqrt{2\pi}}{\sqrt{\beta^2-\nu^2}}\sin\left(\frac{S_1(x_1(\nu),\nu)}{\varepsilon}-\frac{\pi}{4}\right)\right|
	\leq C(\delta).
\]

Let $c_1>0$ be a small parameter such that if  
\(
\left|\cos\left(x\right)\right|\leq c_1,
\)
then 
\(
\d \left(x, \pi\Z+\frac{\pi}{2}\right)\leq \frac{\pi}{4}.
\)
Using~\eqref{eq:approxI-}, there exist $c_2(\delta)>0$ and $\varepsilon_0(\delta)>0$  such that for $\varepsilon<\varepsilon_0(\delta)$, then the inequality
\begin{equation}\label{ineq:condI-}
|I(\varepsilon,\nu)|\leq c_2(\delta)\sqrt{\varepsilon}
\end{equation}
implies
\[
\left|\cos\left(\frac{S_1(x_1(\nu),\nu)}{\varepsilon}-\frac{\pi}{4}\right)\right|\leq c_1.
\]

Let $N\geq 0$ be an integer. Then there exists $\nu_N^0=\nu_N^0(\varepsilon)\geq -\beta$ such that
\begin{equation}\label{cond:nu0}
\int_{-\nu_N^0}^{\beta}F(\eta)\d\eta
-\frac{3\varepsilon}{4}	-\varepsilon N
	=0,
\end{equation}
implying
\[
\cos\left(\frac{S_1(x_1(\nu_N^0),\nu_N^0)}{\varepsilon}-\frac{\pi}{4}\right)=0.\]
As a consequence, inequality~\eqref{eq:approxI-} implies that for $\varepsilon<\varepsilon_0(\delta)$, there holds
\[
|I(\varepsilon,\nu_N^0)|
	\leq C(\delta)\varepsilon
	\leq c_2(\delta)\sqrt{\varepsilon}/2.
\]
Assume that $|\nu_N^0|\leq \beta-2\delta$.  Let $[\nu_*,\nu^*]\subset[-\beta+\delta,\beta-\delta]$ be the largest interval containing $\nu_N^0$ and on which inequality~\eqref{ineq:condI-} holds.
We prove that the interval $[\nu_*,\nu^*]$  encloses exactly one eigenvalue because of monotonicity of $I$ along the parameter $\nu$, and conversely that this interval is large enough to enclose all the eigenvalues associated to $N$.

On $[\nu_*,\nu^*]$, we have by construction
\[
\left|\sin\left(\frac{S_1(x_1(\nu),\nu)}{\varepsilon}-\frac{\pi}{4}\right)\right|\geq\sqrt{1- c_1^2}.\]
Given that $x_1(\nu)=\arccos(-\frac{\nu}{\beta})\geq \frac{1}{C(\delta)}$, we deduce that for $\varepsilon<\varepsilon_0(\delta)$,
\begin{equation}\label{eq:I-deriv}
|\partial_{\nu}I(\varepsilon,\nu)|
	\geq \frac{\sqrt{1-c_1^2}}{C(\delta)\sqrt{\varepsilon}}- C(\delta)
	\geq \frac{1}{C'(\delta)\sqrt{\varepsilon}}.
\end{equation}
By continuity, the $\nu$-derivative of $I(\varepsilon,\nu)$ stays of the same sign on $[\nu_*,\nu^*]$. For instance if $I(\varepsilon,\cdot)$ is increasing and $I(\varepsilon,\nu_N^0)<0$, then there holds for $\nu^*\geq\nu\geq \nu_N^0$ that
\[
\frac{1}{C'(\delta)\sqrt{\varepsilon}}(\nu-\nu_N^0)\leq I(\varepsilon,\nu) - I(\varepsilon,\nu_N^0).
\] 
Since $|I(\varepsilon,\nu_N^0)|\leq c_2(\delta)\sqrt{\varepsilon}/2$, then as long as $I(\varepsilon,\nu)<0$, we know that $I(\varepsilon,\cdot)$ is increasing and inequality~\eqref{ineq:condI-} stays satisfied. Therefore, there exists $\nu_N\in [\nu_*,\nu^*]$ such that $I(\varepsilon,\nu_N)=0$ and $\nu_N$ is a Lax eigenvalue. We also know that $|\nu_N-\nu_N^0|\leq\delta$ for $\varepsilon<\varepsilon_0(\delta)$. An adaptation of this argument applies to the other cases, when $I(\varepsilon,\cdot)$ is decreasing or when $I(\varepsilon,\nu_N^0)\geq 0$.

The integer $N$ is thus uniquely defined in the following inequality, which stays true on the (non ordered) interval $[\nu_N^0,\nu_N]$ by construction
\begin{equation*}\label{eq:defN}
\frac{\pi}{\varepsilon}\left|\int_{-\nu}^{\beta}F(\eta)\d\eta
-\frac{3\varepsilon}{4}	-\varepsilon N\right|
	=\left|\frac{S_1(x_1(\nu),\nu)}{\varepsilon}-\frac{\pi}{4}-\frac{\pi}{2}-\pi N\right|
	\leq \frac{\pi}{4}.
\end{equation*}
Moreover, this is the same integer $N$ on the whole interval $[\nu_N^0,\nu_N]$ by continuity. Given $N$, then $\nu_N$ is uniquely defined in $[\nu_*,\nu^*]$ because $I(\varepsilon,\cdot)$ is strictly monotone in this interval.  

Conversely, the function $\nu\mapsto S_1(x_1(\nu),\nu)=\sqrt{\beta^2-\nu^2}+\nu \arccos(-\nu/\beta)$ is $C_1(\delta)$-Lipschitz on the interval $[-\beta+\delta,\beta-\delta]$. As a consequence, let $C_0(\delta)$ be a large constant to be chosen later and let $\nu$ such that
\[
|\nu-\nu_N^0|\leq \frac{c_2(\delta)\varepsilon}{C_0(\delta)C_1(\delta)}.
\]
For $\varepsilon<\varepsilon_0(\delta)$, the upper bound is less than $\delta$ so that $\nu\in[-\beta+\delta,\beta-\delta]$. 
Moreover, the Lipschitz bound implies
\[
\left|\cos\left(\frac{S_1(x_1(\nu),\nu)}{\varepsilon}-\frac{\pi}{4}\right)\right|
	=\left|\cos\left(\frac{S_1(x_1(\nu),\nu)}{\varepsilon}-\frac{\pi}{4}\right)
	-\cos\left(\frac{S_1(x_1(\nu_N^0),\nu_N^0)}{\varepsilon}-\frac{\pi}{4}\right) \right|
	\leq \frac{c_2(\delta)}{C_0(\delta)}.
\]
Consequently, inequality~\eqref{eq:approxI-} implies that for $\varepsilon<\varepsilon_0(\delta)$ chosen small enough and $C_0(\delta)$ chosen large enough, there holds
\[
|I(\varepsilon,\nu)|\leq  \frac{c_2(\delta)\sqrt{\varepsilon}}{2}+C(\delta)\varepsilon\leq c_2(\delta)\sqrt{\varepsilon}.
\]
Therefore, inequality~\eqref{ineq:condI-} holds true, and we have proven that when  $\frac{c_2(\delta)\varepsilon}{C_0(\delta)C_1(\delta)}<\delta$, then
\[
\left[\nu_N^0-\frac{c_2(\delta)\varepsilon}{C_0(\delta)C_1(\delta)},\nu_N^0+\frac{c_2(\delta)\varepsilon}{C_0(\delta)C_1(\delta)}\right]\subset[\nu_*,\nu^*].
\]
We deduce that the Lax eigenvalue $\nu_N$ associated to the integer $N$, which satisfies $|\nu-\nu_N^0|\leq C(\delta)\varepsilon\sqrt{\varepsilon}$ thanks to inequality~\eqref{ineq:defN-}, must belong to $[\nu_*,\nu^*]$ and is therefore uniquely defined on $[-\beta+\delta,\beta-\delta]$.


To conclude, if $|\nu|\leq \beta-\delta$ is a small eigenvalue, then there exists $|\nu_N^0|\leq \beta-\delta/2$ as above. Conversely, if $|\nu_N^0|\leq \beta-\delta/2$, then one can construct a small eigenvalue $|\nu|\leq\beta-\delta/4$  such that $N=N(\varepsilon,\nu)$, and by restriction to the interval $[-\beta+\delta,\beta-\delta]$, we get all the eigenvalues satisfying $|\nu|\leq\beta-\delta$.

\paragraph{Large eigenvalues}
We now establish the asymptotics for the eigenvalues satisfying $K(\delta)\geq \nu\geq\beta+\delta$. Using the Laplace method,
\[
\left|\partial_{\nu}I_2(\varepsilon,\nu)
	-\frac{x_2(\nu)}{\varepsilon}\frac{\sqrt{2\pi\varepsilon}}{\sqrt{\nu^2-\beta^2}}e^{S_2(x_2(\nu),\nu)/\varepsilon}\right|
	\leq C(\delta),
\]
whereas
\[
\left|\partial_{\nu} I_1(\varepsilon,\nu)\right|
	\leq C(\delta).
\]
We proceed similarly as the small eigenvalue case.

Let $N\geq 0$ and $\nu_N^0=\nu_N^0(\varepsilon)$ such that 
\[\nu_N^0=(N+1)\varepsilon,
\]
so that
\(
\sin\left(\pi\frac{\nu_N^0}{\varepsilon}\right)=0.
\)
We assume that $K(\delta)+\delta\geq \nu_N^0\geq \beta+2\delta$. Let $c_1>0$ such that if $|\sin(x)|\leq c_1$, then $\d(x,\pi\Z)\leq \frac{\pi}{4}$, and using inequality~\eqref{eq:approxI+}, let $c_2(\delta)>0$ such that if
\begin{equation}\label{eq:condI+}
|I(\varepsilon,\nu)|\leq c_2(\delta)\sqrt{\varepsilon},
\end{equation}
then for $\varepsilon<\varepsilon_0(\delta)$, there holds
\[
\left|\sin\left(\pi\frac{\nu}{\varepsilon}\right)\right|\leq c_1.
\]
Let $[\nu_*,\nu^*]\subset[\beta+\delta,K(\delta)+2\delta]$ be the largest interval containing $\nu_N^0$ and on which inequality~\eqref{eq:condI+} holds.

By definition, on this interval, we have
\[
\left|\cos\left(\pi\frac{\nu}{\varepsilon}\right)\right|\geq \sqrt{1-c_1^2}.
\]
Moreover, recall that on $[\beta+\delta,K(\delta)+2\delta]$, we have $S_2(x_2(\nu),\nu)\geq 0$ and $x_2(\nu)=\cosh^{-1}(\frac{\nu}{\beta})\in[ \frac{1}{C(\delta)},C(\delta)]$. Therefore, when $\varepsilon<\varepsilon_0(\delta)$, inequality~\eqref{eq:I2} implies
\[
|I_2(\varepsilon,\nu)|\geq \frac{\sqrt{\varepsilon}}{C(\delta)}e^{S_2(x_2(\nu),\nu)/\varepsilon}-C(\delta)\varepsilon\geq \frac{\sqrt{\varepsilon}}{C'(\delta)}e^{S_2(x_2(\nu),\nu)/\varepsilon}.
\]
As a consequence, we get
\begin{align*}
\left|\partial_{\nu}\left(\Re(I_1(\varepsilon,\nu))-\sin\left(\pi\frac{\nu}{\varepsilon}\right)I_2(\varepsilon,\nu)\right)\right|
	&=\left|\partial_{\nu}\Re(I_1(\varepsilon,\nu))
	-\sin\left(\pi\frac{\nu}{\varepsilon}\right)\partial_\nu I_2(\varepsilon,\nu)
	-\frac{\pi}{\varepsilon}\cos\left(\pi\frac{\nu}{\varepsilon}\right) I_2(\varepsilon,\nu)\right|\\
	&\geq \frac{\sqrt{1-c_1^2}}{C(\delta)\sqrt{\varepsilon}}e^{S_2(x_2(\nu),\nu)/\varepsilon}- \frac{c_1 C(\delta)}{\sqrt{\varepsilon}} -  C(\delta),
\end{align*}
or when we then choose $\varepsilon<\varepsilon_0(\delta)$ and $c_1$ such that for instance $\sqrt{1-c_1^2}\geq c_1/2$,
\begin{equation*}
\left|\partial_{\nu}I(\varepsilon,\nu)\right|
	\geq \frac{1}{C(\delta)\sqrt{\varepsilon}}e^{S_2(x_2(\nu),\nu)/\varepsilon}
	\geq \frac{1}{C(\delta)\sqrt{\varepsilon}}.
\end{equation*}
Therefore, on $[\nu_*,\nu^*]$, the derivative stays of the same sign.

On the other hand, since $\sin(\pi\nu_N^0/\varepsilon)=0$, then inequality~\eqref{eq:approxI+} implies that $|I(\varepsilon,\nu_N^0)|\leq C(\delta)\varepsilon\leq c_2(\delta)\sqrt{\varepsilon}/2$.
We deduce by monotonicity that there exists a unique $\nu_N\in[\nu_*,\nu^*]$ such that  $I(\varepsilon,\nu)=0$. Moreover, since inequality~\eqref{eq:condI+} holds on $[\nu_*,\nu^*]$ by definition, we have $d(\nu,\varepsilon\Z)\leq \frac{\varepsilon}{4}$ on this interval, and by continuity, the same integer $N$ as for $\nu_N^0$ appears in the inequality
\[
\left|\nu_N-(N+1)\varepsilon\right|\leq \frac{1}{4}.
\]

Conversely, let $C_0(\delta)>0$ be a large constant, let $\varepsilon<\varepsilon_0(\delta)$, and let $\nu$ such that
\[
|\nu-\nu_N^0|\leq \frac{c_2(\delta)\varepsilon}{C_0(\delta)}.
\]
Since $\sin$ is $1$-Lipschitz, then
\[
\left|\sin\left(\pi\frac{\nu}{\varepsilon}\right)\right|
	=\left|\sin\left(\pi\frac{\nu}{\varepsilon}\right)-\sin\left(\pi\frac{\nu_N^0}{\varepsilon}\right)\right|
	\leq \pi\frac{c_2(\delta)}{C_0(\delta)}.
\]
Using inequality~\eqref{eq:approxI+}, we deduce that when $\varepsilon<\varepsilon_0(\delta)$, we have 
\[
|I(\varepsilon,\nu)|\leq \frac{c_2(\delta)\sqrt{\varepsilon}}{2}+C(\delta)\varepsilon\leq c_2(\delta)\sqrt{\varepsilon},
\]
and therefore $\nu\in[\nu_*,\nu^*]$: we have proven $[\nu_N^0-\frac{c_2(\delta)\varepsilon}{C_0(\delta)},\nu_N^0+\frac{c_2(\delta)\varepsilon}{C_0(\delta)}]\subset [\nu_*,\nu^*]$ when $\frac{c_2(\delta)\varepsilon}{C_0(\delta)}<\delta$.
The precise asymptotics that a Lax eigenvalue needs to satisfy~\eqref{eq:defN+}
\[
|\nu-(N(\varepsilon,\nu)+1)\varepsilon|\leq C(\delta)\varepsilon\sqrt{\varepsilon}
\]
ensure that given $N$ such that $K(\delta)+\delta\geq \nu_N^0\geq \beta+2\delta$, necessarily $\nu_N\in[\nu_*,\nu^*]$, therefore there is exactly one Lax eigenvalue.

To conclude, if $K(\delta)\geq \nu\geq \beta+\delta$ is a large eigenvalue, then there exists $K(\delta)+\delta\geq \nu_N^0\geq \beta+\delta/2$ as above. Conversely, if $K(\delta)+\delta/2 \geq \nu_N^0\geq \beta+\delta/4$, then there exists exactly one large eigenvalue  associated to $N$ such that $K(\delta)+\delta/2\geq \nu\geq\beta+\delta/2$, and by restriction we get all the Lax eigenvalues satisfying $K(\delta)\geq \nu\geq\beta+\delta$.

\paragraph{Other eigenvalues}
We now establish an upper bound on the number of eigenvalues which do not fit into any of the two categories listed above. 

First, for every
\[
\frac{K(\delta)+\delta}{\varepsilon}\geq N+1\geq \frac{\beta+2\delta}{\varepsilon},
\]
there holds
\[
K(\delta)+\delta\geq \nu_N^0:=(N+1)\varepsilon\geq \beta+2\delta
\]
so that we get a Lax eigenvalue $\nu_N\in[\beta+\delta,K(\delta)+2\delta]$ which satisfies
\[
|\nu_N-(N+1)\varepsilon|\leq \frac{1}{4}.
\]
But we also have the asymptotic expansion as $n\to\infty$
\(
\lambda_n-n\varepsilon\to 0,
\)
more precisely, by definition of $K(\delta)$, for $n\geq K(\delta)/\varepsilon$, one has
\[
|\lambda_n-n\varepsilon|<\delta.
\]
When $\delta<1/4$, since $N$ is uniquely defined, this implies that the $n$-th Lax eigenvalue satisfies $\lambda_n=\nu_n-\varepsilon$.

We now establish a lower bound on the number of small Lax eigenvalues $\lambda$ satisfying $\lambda+\varepsilon\in[-\beta+\delta,\beta-\delta]$. Condition \eqref{cond:nu0}
\begin{equation*}
\int_{-\nu_N^0}^{\beta}F(\eta)\d\eta
-\frac{3\varepsilon}{4}	-\varepsilon N
	=0
\end{equation*}
is true for some $\nu_N^0\in[-\beta+2\delta,\beta-2\delta]$  as soon as
\begin{equation*}
\int_{\beta-2\delta}^{\beta}F(\eta)\d\eta
	\leq\varepsilon N+\frac{3\varepsilon}{4}
	\leq \int_{-\beta+2\delta}^{\beta}F(\eta)\d\eta.
\end{equation*}
We deduce that there are at least
\[
\frac{1}{\varepsilon} \int_{-\beta+2\delta}^{\beta-2\delta}F(\eta)\d\eta-1
\]
suitable integers $N$. But since $\int_{-\beta}^{\beta}F(\eta)\d\eta=\beta$ and $F(\eta)\leq 1$ for every $\eta$, we have
\[
\frac{1}{\varepsilon}\left|\int_{-\beta+2\delta}^{\beta-2\delta}F(\eta)\d\eta
	-\beta\right|
	\leq \frac{\beta\delta}{\varepsilon}.
\]

As a conclusion, among the integers $N$ which do not lead to large eigenvalues $\nu_N\geq\beta+\delta$, that is, which satisfy
\[
N\leq \frac{\beta+2\delta}{\varepsilon}-1,
\]
there are at least $\frac{\beta-C\delta}{\varepsilon}$ of them which lead to small eigenvalues $|\nu_N|\leq\beta-\delta$. The remaining eigenvalues consist of no more than $\frac{C\delta}{\varepsilon}$ indexes.


\begin{proof}[Proof of Theorem~\ref{thm:asymptotics}]
When $\lambda_n+\varepsilon,\lambda_p+\varepsilon\in\Lambda_+(\delta)$, we have already seen from the study of the large eigenvalues that if $n\leq K(\delta)/\varepsilon$, then
\(
|\lambda_n-n\varepsilon|\leq C(\delta)\varepsilon\sqrt{\varepsilon},
\)
whereas  if $n> K(\delta)/\varepsilon$, then
\(
|\lambda_n-n\varepsilon|\leq C\delta.
\)

In the former part, we have seen that there are at most $\frac{C\delta}{\varepsilon}$ eigenvalues such that $\lambda_n+\varepsilon\not\in\Lambda_-(\delta)\cup\Lambda_+(\delta)$.
Reasoning with $\delta/2$ instead of $\delta$, the indexes counting argument from the former part implies the index $N$ leads to an eigenvalue $\nu_N=\lambda_n+\varepsilon\in[\beta-\delta,\beta-\delta/2]$ as soon as
\begin{equation*}
\int_{-\beta+\delta/4}^{\beta}F(\eta)\d\eta
	\leq\varepsilon N+\frac{3\varepsilon}{4}
	\leq \int_{-\beta+\delta/2}^{\beta}F(\eta)\d\eta.
\end{equation*}
Since $F\geq 1/C$ on $(-\infty,0]$, we deduce that that there are at least 
\[
\frac{1}{\varepsilon}\int_{-\beta+\delta/4}^{-\beta+\delta/2}F(\eta)\d\eta\geq \frac{\delta}{C\varepsilon}
\]
such indexes, or at least $\frac{\delta}{C\varepsilon}$ Lax eigenvalues $\lambda_n+\varepsilon\in[\beta-\delta,\beta-\delta/2]$. Similarly, one knows that there are at least $\frac{\delta}{C\varepsilon}$ Lax eigenvalues such that $\lambda_n+\varepsilon\in[-\beta+\delta/2,-\beta+\delta]$.

Finally, in the region $\lambda_n+\varepsilon\in\Lambda_-(\delta)=[-\beta+\delta,\beta-\delta]$, the counting of the indexes leads to the same conclusion, except that only the lower bound $F(\eta)\geq1/C(\delta)$ holds on $(-\infty,\beta-\delta/4]$, so that there are between $\frac{1}{C(\delta)\varepsilon}$ and $\frac{C\delta}{\varepsilon}$ eigenvalues in the region $[-\beta,-\beta+\delta]$. Therefore there exists $\frac{1}{C(\delta)\varepsilon}\leq N_0\leq \frac{C\delta}{\varepsilon}$ such that for every $n$, 
if $\lambda_n+\varepsilon=\nu_N\in\Lambda_-(\delta)$, then $n-N= N_0$. Inequality~\eqref{ineq:defN-} becomes
\begin{equation*}
\left|\int_{-\lambda-\varepsilon}^{\beta}F(\eta)\d\eta
-\frac{3\varepsilon}{4}
	-\varepsilon (n-N_0)\right|
	\leq C(\delta)\varepsilon\sqrt{\varepsilon}
\end{equation*}
for every $\lambda_n+\varepsilon\in\Lambda_-(\delta)$.
This implies the small eigenvalues point of the Theorem.
\end{proof}

\bibliography{mybib.bib}{}

\begin{thebibliography}{10}

\bibitem{Benjamin1967}
T.~B. Benjamin.
\newblock {Internal waves of permanent form in fluids of great depth}.
\newblock {\em Journal of Fluid Mechanics}, 29(3):559–592, 1967.

\bibitem{ClaeysGrava2009universality}
T.~Claeys and T.~Grava.
\newblock {Universality of the break-up profile for the KdV equation in the
  small dispersion limit using the Riemann-Hilbert approach}.
\newblock {\em Communications in Mathematical Physics}, 286(3):979--1009, 2009.

\bibitem{ClaeysGrava2010painleve}
T.~Claeys and T.~Grava.
\newblock {Painlev{\'e} II asymptotics near the leading edge of the oscillatory
  zone for the Korteweg—de Vries equation in the small-dispersion limit}.
\newblock {\em Communications on Pure and Applied Mathematics}, 63(2):203--232,
  2010.

\bibitem{ClaeysGrava2010solitonic}
T.~Claeys and T.~Grava.
\newblock {Solitonic asymptotics for the Korteweg--de Vries equation in the
  small dispersion limit}.
\newblock {\em SIAM journal on mathematical analysis}, 42(5):2132--2154, 2010.

\bibitem{CoifmanWickerhauser1990}
R.~R. Coifman and M.~V. Wickerhauser.
\newblock {The scattering transform for the Benjamin-Ono equation}.
\newblock {\em Inverse Problems}, 6(5):825, 1990.

\bibitem{DeiftVenakidesZhou1997}
P.~Deift, S.~Venakides, and X.~Zhou.
\newblock {New results in small dispersion KdV by an extension of the steepest
  descent method for Riemann-Hilbert problems}.
\newblock {\em International Mathematics Research Notices}, 1997(6):285--299,
  1997.

\bibitem{DengBiondiniTrillo2016}
G.~Deng, G.~Biondini, and S.~Trillo.
\newblock {Small dispersion limit of the Korteweg--de Vries equation with
  periodic initial conditions and analytical description of the
  Zabusky--Kruskal experiment}.
\newblock {\em Physica D: Nonlinear Phenomena}, 333:137--147, 2016.

\bibitem{DzyubenkoPleshakov2008}
G.~Dzyubenko and M.~Pleshakov.
\newblock {Comonotone approximation of periodic functions}.
\newblock {\em Mathematical Notes}, 83(1):180--189, 2008.

\bibitem{Faraut2021}
J.~Faraut.
\newblock {\em {Calcul int{\'e}gral}}.
\newblock EDP sciences, 2021.

\bibitem{AblowitzFokas1983}
A.~Fokas and M.~Ablowitz.
\newblock {The inverse scattering transform for the Benjamin-Ono equation—a
  pivot to multidimensional problems}.
\newblock {\em Studies in Applied Mathematics}, 68(1):1--10, 2 1983.

\bibitem{bo_hierarchie}
L.~Gassot.
\newblock {The third order Benjamin-Ono equation on the torus: well-posedness,
  traveling waves and stability}.
\newblock {\em Annales de l'Institut Henri Poincaré C, Analyse non linéaire},
  38(3):815--840, 2021.

\bibitem{Gerard2019}
P.~G{\'e}rard.
\newblock {A nonlinear Fourier transform for the Benjamin--Ono equation on the
  torus and applications}.
\newblock {\em S{\'e}minaire Laurent Schwartz—EDP et applications}, pages
  1--19, 2019-2020.

\bibitem{GerardKappeler2019}
P.~G{\'e}rard and T.~Kappeler.
\newblock {On the integrability of the Benjamin-Ono equation on the torus}.
\newblock {\em {Communications on Pure and Applied Mathematics}}, 2020.

\bibitem{GerardKappelerTopalov2020-2}
P.~G{\'e}rard, T.~Kappeler, and P.~Topalov.
\newblock {On the spectrum of the Lax operator of the Benjamin-Ono equation on
  the torus}.
\newblock {\em {Journal of Functional Analysis}}, 279(12):108762, 2020.

\bibitem{GerardKappelerTopalov2020}
P.~G{\'e}rard, T.~Kappeler, and P.~Topalov.
\newblock {Sharp well-posedness results of the Benjamin-Ono equation in
  \(H^s(\mathbb{T},\mathbb{R})\) and qualitative properties of its solution}.
\newblock {\em to appear in Acta Mathematica, arxiv:2004.04857}, 2020.

\bibitem{GerardKappelerTopalov2021}
P.~G{\'e}rard, T.~Kappeler, and P.~Topalov.
\newblock {On the analytic Birkhoff normal form of the Benjamin-Ono equation
  and applications}.
\newblock {\em to appear in Nonlinearity, arxiv:2103.07981}, 2021.

\bibitem{GerardKappelerTopalov2021analyticity}
P.~G{\'e}rard, T.~Kappeler, and P.~Topalov.
\newblock {On the analyticity of the nonlinear Fourier transform of the
  Benjamin-Ono equation on $\mathbb{T}$}, 2021.

\bibitem{Hormander2015}
L.~H{\"o}rmander.
\newblock {\em {The analysis of linear partial differential operators I:
  Distribution theory and Fourier analysis}}.
\newblock Springer, 2015.

\bibitem{KaupMatsuno1998}
D.~Kaup and Y.~Matsuno.
\newblock {The inverse scattering transform for the Benjamin--Ono equation}.
\newblock {\em Studies in applied mathematics}, 101(1):73--98, 1998.

\bibitem{LaxLevermore}
P.~D. Lax and C.~David~Levermore.
\newblock {The small dispersion limit of the Korteweg-de Vries equation.}
\newblock {\em Communications on Pure and Applied Mathematics},
  36(3):253--290(I), 571--593(II), 809--829(III), 1983.

\bibitem{LorentzZeller1968}
G.~Lorentz and K.~Zeller.
\newblock {Degree of approximation by monotone polynomials I}.
\newblock {\em Journal of Approximation Theory}, 1(4):501--504, 1968.

\bibitem{Matsuno1998}
Y.~Matsuno.
\newblock {Nonlinear modulation of periodic waves in the small dispersion limit
  of the Benjamin-Ono equation}.
\newblock {\em Physical Review E}, 58(6):7934, 1998.

\bibitem{MillerWetzel2016rational}
P.~D. Miller and A.~N. Wetzel.
\newblock {Direct Scattering for the Benjamin--Ono Equation with Rational
  Initial Data}.
\newblock {\em Studies in Applied Mathematics}, 137(1):53--69, 2016.

\bibitem{MillerWetzel2016}
P.~D. Miller and A.~N. Wetzel.
\newblock {The scattering transform for the Benjamin--Ono equation in the
  small-dispersion limit}.
\newblock {\em Physica D: Nonlinear Phenomena}, 333:185--199, 2016.

\bibitem{MillerXu2011}
P.~D. Miller and Z.~Xu.
\newblock {On the zero-dispersion limit of the Benjamin-Ono Cauchy problem for
  positive initial data}.
\newblock {\em Communications on Pure and Applied Mathematics}, 64(2):205--270,
  2011.

\bibitem{MillerXu2011hierarchy}
P.~D. Miller and Z.~Xu.
\newblock {The Benjamin-Ono hierarchy with asymptotically reflectionless
  initial data in the zero-dispersion limit}.
\newblock {\em Communications in Mathematical Sciences}, 10(1):117--130, 2012.

\bibitem{Moll2019-2}
A.~Moll.
\newblock {Exact Bohr-Sommerfeld conditions for the quantum periodic
  Benjamin-Ono equation}.
\newblock {\em SIGMA}, 15(098):1--27, 2019.

\bibitem{Moll2019-1}
A.~Moll.
\newblock {Finite gap conditions and small dispersion asymptotics for the
  classical periodic Benjamin-Ono equation}.
\newblock {\em Quart. Appl. Math.}, 78:671--702, 2020.

\bibitem{Ono1977}
H.~Ono.
\newblock {Algebraic solitary waves in stratified fluids}.
\newblock {\em Journal of the Physical Society of Japan}, 39(4):1082--1091,
  1975.

\bibitem{Sun2020}
R.~Sun.
\newblock {Complete integrability of the Benjamin--Ono equation on the
  multi-soliton manifolds}.
\newblock {\em Communications in Mathematical Physics}, 383:1051--1092, 4 2021.

\bibitem{Venakides1985}
S.~Venakides.
\newblock {The zero dispersion of the Korteweg-de Vries equation for initial
  potentials with non-trivial reflection coefficient}.
\newblock {\em Communications on Pure and Applied Mathematics}, 38(2):125--155,
  1985.

\bibitem{Venakides1987}
S.~Venakides.
\newblock {The zero dispersion limit of the Korteweg-de Vries equation with
  periodic initial data}.
\newblock {\em Transactions of the American Mathematical Society}, pages
  189--226, 1987.

\bibitem{Venakides1991}
S.~Venakides.
\newblock {The Korteweg-de Vries equation with small dispersion: higher order
  Lax-Levermore theory}.
\newblock In {\em Applied and Industrial Mathematics}, pages 255--262.
  Springer, 1991.

\bibitem{Wu2016}
Y.~Wu.
\newblock {Simplicity and Finiteness of Discrete Spectrum of the Benjamin--Ono
  Scattering Operator}.
\newblock {\em SIAM Journal on Mathematical Analysis}, 48(2):1348--1367, 2016.

\bibitem{Wu2017}
Y.~Wu.
\newblock {Jost Solutions and the Direct Scattering Problem of the
  Benjamin--Ono Equation}.
\newblock {\em SIAM Journal on Mathematical Analysis}, 49(6):5158--5206, 2017.

\bibitem{ZabuskyKruskal1965}
N.~J. Zabusky and M.~D. Kruskal.
\newblock {Interaction of solitons in a collisionless plasma and the recurrence
  of initial states}.
\newblock {\em Physical review letters}, 15(6):240, 1965.

\end{thebibliography}
\bibliographystyle{abbrv}
\Addresses

\end{document}